\documentclass[11pt]{article}

\usepackage{amssymb, amsmath, amsthm, mathtools, amsfonts, times, amsthm, amscd, setspace, leftidx, enumerate, xfrac}
\usepackage{geometry}
\input xy 
\xyoption{all}

\usepackage{tikz-cd}
\usepackage{xcolor}

\usepackage{pdflscape}
\usepackage{adjustbox}

\usepackage[titletoc,title]{appendix}

\usepackage{graphicx}
\graphicspath{ {Graphics/} }
\usepackage{capt-of}

\theoremstyle{plain}

\newtheorem{proposition}{Proposition}[section]
\newtheorem{lemma}[proposition]{Lemma}
\newtheorem{thm}[proposition]{Theorem}
\newtheorem{corollary}[proposition]{Corollary}

\theoremstyle{definition}
\newtheorem{defi}{Definition}[section]

\newtheorem{ex}{Example}[section]

\newcommand{\tn}[0]{\otimes}

\newcommand{\ct}[1]{\mathcal{#1}}
\newcommand{\ov}[1]{\overline{#1}}

\newcommand{\fr}[1]{\mathfrak{#1}}

\newcommand{\SL}[0]{\mathrm{SupLat}}

\newcommand{\id}[0]{\mathrm{id}}

\newcommand{\Rr}[0]{\mathcal{R}}

\newcommand{\st}[1]{\lbrace #1 \rbrace}
\newcommand{\Pl}[1]{\mathcal{P}(#1)}

\newcommand{\tht}[4]{\prescript{\Theta_{#1}({#2},{#3})}{}{{#4}}}

\allowdisplaybreaks
\begin{document}

\title{Drinfeld Twists on Skew Braces}
\author{Aryan Ghobadi \\ \small{Queen Mary University of London }\\\small{ School of Mathematics, Mile End Road}\\\small{ London E1 4NS, UK }\\ \small{Email: a.ghobadi@qmul.ac.uk}}
\date{}

\maketitle
\begin{abstract}
We introduce the notion of Drinfeld twists for both set-theoretical YBE solutions and skew braces. We give examples of such twists and show that all twists between skew braces come from families of isomorphisms between their additive groups. We then describe the relation between these definitions and co-twists on FRT-type Hopf algebras in the category $\SL$, and prove that any co-twist on a co-quasitriangular Hopf algebra in $\SL$ induces a Drinfeld twist on its remnant skew brace. We go on to classify co-twists on bicrossproduct Hopf algebras coming from groups with unique factorisation, and the twists which they induce on skew braces.
\end{abstract}

\begin{footnotesize}2020\textit{ Mathematics Subject Classification}: 16T25, 18M15, 16T99, 17B37
 \end{footnotesize}

\section{Introduction}\label{S:Intro}
Set-theoretical solutions of the Yang-Baxter equation (YBE) and their classification have attracted a lot of attention in recent years \cite{bachiller2018solutions,doikou2020set,doikou2021set,guarnieri2017skew,smoktunowicz2018skew}. A prominent role in this theory is played by \emph{skew braces}, which are sets with two compatible group structures \cite{guarnieri2017skew}. Each set-theoretical YBE solution has a corresponding universal skew brace, while any skew brace comes equipped with a set-theoretical YBE solution. In fact, any skew brace is equivalent to a group with a compatible YBE solution on it, called a braiding operator. In the recent work \cite{ghobadi2020skew}, we showed that skew braces are in fact remainders of co-quasitriangular Hopf algebras (CQHA) in the category $\SL$, which has complete lattices as objects and join-preserving maps as morphisms. Given any CQHA in this category we can construct a group called the remnant of the Hopf algebra, which comes equipped with a braiding operator. Conversely, any group with a braiding operator appears as the remnant of a CQHA in $\SL$, but not necessarily of a unique one. The first fruits of this theory already appeared in \cite{ghobadi2020skew}, where the correspondence between skew braces and groups with braiding operators was described as the shadow of the theory of transmutation for CQHAs \cite{majid1993transmutation}. In addition to explaining aspects of the theory of skew braces, this machinery enables us to apply standard methods in the theory of Hopf algebras to produce new skew braces and set-theoretical YBE solutions. In this work, we will introduce the theory of Drinfeld twists for skew braces using this machinery. 
 
Drinfeld twists were originally introduced in \cite{drinfeld1989quasi,drinfeld1989quasi2} in the study of quasi-Hopf algebras. Given a Drinfeld twist on an ordinary Hopf algebra, we can produce a new Hopf algebra structure on the same space by conjugating the coalgebra structure by the twist element, see \cite{majid2000foundations}. Moreover, if the Hopf algebra is equipped with a quasitriangular structure, the twisted Hopf algebra also obtains a quasitriangular structure. The dual of this theory also works and is utilised under the name of Drinfeld co-twists \cite{majid1992tannaka}. By applying this theory in the category of $\SL$, we see that any co-twist on a CQHA in $\SL$ produces a new skew brace structure on its remnant. The study of co-twist on the FRT-type Hopf algebras constructed in \cite{ghobadi2020skew}, then reveals how we can formulate a novel notion of twists directly on skew braces. 

In Section~\ref{S:DriTwis}, we introduce the notion of Drinfeld twists for skew braces, Definitions \ref{D:TwistSkw}, and show directly how such a twist induces new skew brace structure on the same set, Theorem \ref{T:TwistSkw}. The analogues definition and result for twists on set-theoretical YBE solutions also appear in Definition \ref{D:TwistYBE} and Theorem \ref{T:TwistYBE}. In Example \ref{E:Doikou}, we show that any involutive YBE solution can be obtained by a twist on the permutation solution, while in Example \ref{E:Theta} we show that any skew brace $(G,.,\star)$ can be obtained by a twist on the trivial skew brace, $(G,\star ,\star)$, defined by its additive group $(G,\star)$. We also comment on the difference between Drinfeld twists on the universal skew brace of a YBE solution and twists on the solution itself in Example \ref{E:Universal}. Beyond examples, in Theorem \ref{T:YBEGroupoid}, we show how to compose and invert twists, and how one can view a twist on a YBE solution as a morphism going from the original solution to the new solution obtained by twisting. The same techniques follow for twists on skew braces, Theorem \ref{T:SkwGroupoid}, and we obtain a groupoid which has skew brace as its objects and Drinfeld twists as morphisms. This result together with Example \ref{E:Theta} allow us to translate the problem of whether two skew braces are related via Drinfeld twists fully in terms of their additive groups, Theorem \ref{T:GrpTheoretic}. Finally, in Corollary \ref{C:Decomposition} we conclude that twists between a pair of skew braces $(G,.,\star)$ and $(G,.',\star')$ correspond to families of group isomorphisms $\lbrace f_{g}:(G,\star)\rightarrow (G,\star') \rbrace_{g\in G}$ satisfying $f_{g}(g)=g$. 

While the theory of Drinfeld twists on skew braces can be understood independently, it is inspired by the Hopf algebraic point of view in $\SL$. Our main results extending the work in \cite{ghobadi2020skew} appear in Section \ref{S:SL}. As mentioned earlier, studying co-twists on the FRT-type algebras constructed in \cite{ghobadi2020skew}, reveal the definition of twists on YBE solutions and skew braces directly, Theorem \ref{T:TwistHw}. More generally, we show that a Drinfeld co-twist on any arbitrary CQHA in $\SL$ induces a Drinfeld twist on its remnant skew brace, Theorem \ref{T:AnyCQHA}. In Section \ref{S:Groups}, we classify co-twists on dualizable CQHAs in $\SL$, which correspond to matched pairs of groups \cite{takeuchi1981matched}, and thereby describe the induced twists coming from these structures, Theorem \ref{T:TwistG+G-}. Since any skew braces gives rise to such a pair of groups in a trivial way, we obtain an explicit family of twists for any given skew brace in Corollary \ref{C:G-}. We conclude the manuscript by presenting a number of open problems regarding the theory of Drinfeld twists on skew braces in Section \ref{S:Final}.

Lastly, let us comment on the similarities of our work with other sources which study linear YBE solutions. In \cite{kulish2000twisting}, the authors present the notion of twist on linear YBE solutions, by studying Drinfeld twists on the FRT bialgebras. Our definition of Drinfeld twists on set-theoretical YBE solutions and its relation to the $\SL$-FRT algebras can be seen as the set-theoretical version of their work. In particular, the proof of Theorem \ref{T:TwistYBE} is practically the same as Proposition 1 in \cite{kulish2000twisting}. However, the introduction of twists on skew braces is the truly novel portion of our work. A discussion relating Drinfeld twists and braces also appears \cite{doikou2020set,doikou2021set}, where the authors study Drinfeld twists on linear FRT-type algebras coming from the linearisation of involutive YBE solutions. We are able to rephrase part of the work done in \cite{doikou2021set} in terms of our notion of Drinfeld twists in Example \ref{E:Doikou}.

\textbf{Notation.} Throughout this work, $X$ will always denote a set and indices on a map such as $r_{ij}: X^{3}\rightarrow X^{3}$ will denote the application of $r$ to the $i$ and $j$-th components of $X^{3}$ e.g. $r_{23}= \id_{X} \times r$. We use the standard notation $r(x,y)=(\sigma_{x}(y),\gamma_{y}(x))$ for maps $\sigma_{x},\gamma_{y}:X\rightarrow X$. The flip map on the set $X^{2}$ which sends a pair $(x,y)$ to $(y,x)$ will always be denoted by $\fr{fl}_{X}$. For any set $S$, $\id_{S}$ denotes the identity map. By a groupoid we mean a category where all morphisms are invertible. The triple $(G,m,e)$ will always denote a group, where $e\in G$ is the identity element and $m$ the multiplication operation. We will use $m$ and $.$ interchangably in calculations. The term skew brace will always refer to skew left brace as defined in \cite{guarnieri2017skew}. In Sections \ref{S:SL} and \ref{S:Groups} we adapt the notation of \cite{ghobadi2020skew} and \cite{LYZ2}, respectively. 

\textbf{Acknowledgements.} The author would like to thank Shahn Majid for his helpful comments on an earlier draft of this work.
\section{Twists on Skew Braces}\label{S:DriTwis}
In this section, we introduce the notion of Drinfeld twists on set-theoretical YBE solutions and skew braces. While twists on YBE solutions first appeared in \cite{kulish2000twisting}, the definition of twists on skew braces, Definition \ref{D:TwistSkw}, is new and its relation with Drinfeld twists on Hopf algebras will be explored in Section \ref{S:SL}. In this section, we show how these twists can be viewed as morphisms between YBE solutions and skew braces, respectively, and determine when two skew braces can be related via such twists. Note that we only consider the braid form of the Yang-Baxter equation and by a set-theoretical YBE solution we mean a set $X$ and an \emph{invertible} map $r:X^{2}\rightarrow X^{2}$ satisfying $r_{23} r_{12} r_{23}= r_{12} r_{23} r_{12}$. 
\begin{defi}\label{D:TwistYBE} A \emph{Drinfeld twist} on a set-theoretical YBE solution $(X,r)$ consists of a triple of invertible maps $F:X^{2}\rightarrow X^{2}$ and $\Phi,\Psi:X^{3}\rightarrow X^{3}$ satisfying
\begin{align}
\label{Eq:T1} (F\times \id_{X})\Psi &= (\id_{X}\times F) \Phi \tag{T1}
\\\label{Eq:T2} \Phi(\id_{X}\times r) &=(\id_{X}\times r)\Phi \tag{T2}
\\\label{Eq:T3} \Psi(r\times\id_{X} ) &= (r\times \id_{X})\Psi\tag{T3}
\end{align}
\end{defi}
The notion of twists on YBE solutions has previously appeared in \cite{kulish2000twisting}, but is written in terms of $(F^{-1},\Phi^{-1},\Psi^{-1})$. As observed in \cite{kulish2000twisting}, the triple $(F,\Phi,\Psi)$ is determined uniquely by the pair $(F,G)$, where $G= F_{12}\Psi=F_{23}\Phi$ and we could re-write the conditions of Definition \ref{D:TwistYBE} for this pair, but in application this will only lengthen our proofs.

\begin{thm}\label{T:TwistYBE} Given a Drinfeld twist $(F,\Phi,\Psi)$ on a set-theoretical YBE solution $(X,r)$, we obtain a new YBE solution, $(X,FrF^{-1})$.
\end{thm}
\begin{proof} For this proof we observe the simple fact that \eqref{Eq:T1} implies $F_{23}^{-1}F_{12}= \Phi\Psi^{-1}$. We now show that $R_{23} R_{12} R_{23}= R_{12} R_{23} R_{12}$ holds for $R=FrF^{-1} $ by simplifying both sides: 
\begin{align*}
R_{23} R_{12} R_{23}&=F_{23}r_{23}F_{23}^{-1}F_{12}r_{12}F_{12}^{-1}F_{23}r_{23}F_{23}^{-1} = F_{23}r_{23}\Phi\Psi^{-1}r_{12}F_{12}^{-1}F_{23}r_{23}F_{23}^{-1}
\\&= F_{23}\Phi r_{23} r_{12}\Psi^{-1}F_{12}^{-1}F_{23}r_{23}F_{23}^{-1}= F_{23}\Phi r_{23} r_{12}\Phi^{-1} r_{23}F_{23}^{-1}= F_{23}\Phi r_{23} r_{12} r_{23}\Phi^{-1} F_{23}^{-1}
\\ R_{12} R_{23} R_{12}&=F_{12}r_{12}F_{12}^{-1}F_{23}r_{23}F_{23}^{-1}F_{12}r_{12}F_{12}^{-1}= F_{12}r_{12}F_{12}^{-1}F_{23}r_{23}\Phi\Psi^{-1}r_{12}F_{12}^{-1}
 \\&= F_{12}r_{12}F_{12}^{-1}F_{23}\Phi r_{23} r_{12}\Psi^{-1}F_{12}^{-1}= F_{12}r_{12}\Psi r_{23} r_{12}\Psi^{-1}F_{12}^{-1}=F_{12}\Psi r_{12} r_{23} r_{12}\Psi^{-1}F_{12}^{-1}
\end{align*}
Since $r$ satisfy the YBE and \eqref{Eq:T1} is equivalent to $F_{12}\Psi = F_{23}\Phi$, the two sides are equal. 
\end{proof}

Observe that $FrF^{-1}$ is involutive i.e. satisfies $(FrF^{-1})^{2}=\id_{X^{2}}$ if and only if $r$ is involutive. In fact,  the statements $r^{n}= \id_{X^{2}}$ and $(FrF^{-1})^{n}=\id_{X^{2}}$ are equivalent for any natural number $n\in \mathbb{N}$. Recall, that a YBE solution $r$ is said to be non-degenerate if $\sigma_{x},\gamma_{y}$ are bijections for all $x,y\in X$. For $FrF^{-1}$ to also be non-degenerate we need $f_{x}, f'_{y}$ to be bijections, where $F(x,y)=(f_{y}(x),f'_{x}(y))$. 

We can view a Drinfeld twist $(F,\Phi,\Psi)$ on a solution $(X,r)$ as an ``arrow" from the solution $(X,r)$ to $(X,Fr^{-1}F)$. From this perspective, we obtain a category which has set-theoretical YBE solution as objects and Drinfeld twists as morphisms:
\begin{thm}\label{T:YBEGroupoid} The class of set-theoretical YBE solutions, together with Drinfeld twists viewed as morphisms between solution form a groupoid. 
\end{thm} 
\begin{proof} Each YBE solution has an associated trivial twist $(\id_{X^{2}},\id_{X^{3}},\id_{X^{3}})$ which acts as an identity morphism. For composition of twists, let $(F,\Phi,\Psi)$ be a twist on $(X,r)$ and $(G,\phi,\psi)$ be a twist on $(X,FrF^{-1})$. We define the composition of these twists to be the triple 
$$(GF, F_{23}^{-1}\phi F_{23}\Phi, F_{12}^{-1}\psi F_{12}\Psi)$$ 
It is straightforward to check that this triple defines a twist on $r$:
\begin{align*}
\text{(T1): }&\quad (GF)_{23}F_{23}^{-1}\phi F_{23}\Phi = G_{23}F_{23}F_{23}^{-1}\phi F_{12}\Psi = G_{12}\psi F_{12}\Psi= (GF)_{12}F_{12}^{-1}\psi F_{12}\Psi
\\ \text{(T2): }&\quad r_{23}F_{23}^{-1}\phi F_{23}\Phi= F_{23}^{-1}(F_{23}r_{23}F_{23}^{-1})\phi F_{23}\Phi= F_{23}^{-1}\phi(F_{23}r_{23}F_{23}^{-1}) F_{23}\Phi= F_{23}^{-1}\phi F_{23}\Phi r_{23}
\end{align*}
A symmetric argument demonstrates \eqref{Eq:T3}. Finally, we observe that Drinfeld twists are invertible. If $(X,R)$ is obtained by a twist $(F,\Psi, \Phi)$ on $(X,r)$, then the triple $(F^{-1}, F_{23}\Phi^{-1} F^{-1}_{23},F_{12}\Psi^{-1} F^{-1}_{12})$ form a twist on $R$, which recovers $r= F^{-1}RF$ by Theorem \ref{T:TwistYBE}. One can easily check that the latter triple becomes the inverse of $(F,\Psi, \Phi)$, with respect to the composition which we have defined. 
\end{proof}
By Theorem \ref{T:YBEGroupoid}, we can say two YBE solutions are \emph{related} by twists, since if one is obtained as a twist of the other, than the converse statement also holds true. Note that Drinfled twists only send a YBE solution on a set $X$ to another solution on the same underlying set $X$. The important conclusion of Theorem \ref{T:YBEGroupoid} is that any solution in a connected component of the described groupoid can be obtained as a twist of another solution in the same component. In other words, if solutions $r_{1}$ and $r_{2}$ on a set $X$ are related to a solution $(X,r_{3})$ via twists, then $r_{1}$ and $r_{2}$ are also related via Drinfeld twists. 
\begin{ex}\label{E:Doikou} In \cite{doikou2021set}, it is stated that any involutive YBE solution can be viewed as a \emph{quasi-admissible twist} of the permutation solution. Although \cite{doikou2021set} is discussing linear Hopf algebras built on these solutions, the proof of Proposition 3.13 of \cite{doikou2021set} boils down to the following observation: if $r(x,y)= (\sigma_{x}(y),\gamma_{y}(x))$ satisfies the YBE we can introduce $F(x,y)=(x,\sigma_{x}(y))$, $\Phi (x,y,z)= (x,\sigma_{x}(y), \sigma_{\gamma_{y}(x)}(z))$ and $\Psi(x,y,z)=(x,y,\sigma_{x}(\sigma_{y}(z)))$, so that the triple $(F,\Phi,\Psi)$ define a Drinfeld twist on $(X,r)$. Conditions \eqref{Eq:T1}, \eqref{Eq:T2} and \eqref{Eq:T3}, all follow directly from $r$ being a YBE solution and the identities $\sigma_{\sigma_{x}(y)}(\sigma_{\gamma_{x}(y)}(z))= \sigma_{x}(\sigma_{y}(z))$ and $\sigma_{\gamma_{\sigma_{y}(z)}(x)}(\gamma_{z}(y))= \gamma_{\sigma_{x}(y)}(\sigma_{\gamma_{y}(x)}(z)) $. Consequently, by Theorem \ref{T:TwistYBE} we obtain a new YBE solution defined by $FrF^{-1}(x,y)= (y, \sigma_{y}(\gamma_{\sigma^{-1}_{x} (y)}(x)))$. Observe that $\sigma_{y}(\gamma_{\sigma^{-1}_{x} (y)}(x))$ appears in the first component of $r^{2} (x,\sigma^{-1}_{x}(y))$. Going back to the involutive setting considered in \cite{doikou2021set}, we have that $FrF^{-1}(x,y)= (y,x)$, since $r^{2}=\id_{X^{2}}$. Therefore, the permutation solution $(X,\fr{fl}_{X})$ is obtained from a twist on $(X,r)$, and by Theorem \ref{T:YBEGroupoid}, $(X,r)$ can be obtained by the inverse twist on $(X,\fr{fl}_{X})$. 
\end{ex}
By the above example, any involutive solution on $X$ is obtained as a twist of the permutation solution and by Theorem \ref{T:YBEGroupoid}, we can conclude the following result. 
\begin{corollary}\label{C:Inv} If $X$ is a set and $r_{1}$ and $r_{2}$ are involutive YBE solutions in $X$, then there exists a twist $(F,\Phi,\Psi)$ on $r_{1}$ such that $r_{2}=Fr_{1}F^{-1}$. 
\end{corollary}
\begin{ex}\label{E:S4} Let us consider the non-involutive YBE solution from Example 3.4 of \cite{smoktunowicz2018skew}. We take $X=\lbrace 1, 2, 3, 4\rbrace $ and $r(x,y)= (\sigma (y), \gamma (x))$, where $\sigma= (12)$ and $\gamma = (34)$, where $(12),(34)\in S_{4}$ are 2-cycle permutations with standard notation. Consider $F(x,y) =(\sigma(x), \gamma (y))$ with $\Phi(x,y,z)= (\gamma (\sigma( x)),\sigma(y),\sigma (z))$ and $\Psi(x,y,z) = (\gamma(x), \gamma (y) , \gamma (\sigma( z)))$. The twisted YBE solution we obtain is $FrF^{-1}(x,y)= (\gamma(y),\sigma(x))$. This can be generalised to any YBE solution of the form $r(x,y)=(\sigma(x), \gamma (y))$, where bijections $\sigma$ and $\gamma$ commute and $\sigma^{2}=\gamma^{2}=\id_{X}$. 
\end{ex}

\begin{ex}\label{E:noninvLyuba} For any YBE solution $r(x,y)=(\sigma(y),\gamma(x) )$ we can introduce a twist in the style of Example \ref{E:Doikou}, $F(x,y)=(x,\kappa (y))$, where $\kappa$ is a bijection on $X$ which commutes with $\sigma $ and $\gamma$ and we define $\Phi (x,y,z)= (x,\kappa (y), \kappa (z))$ and $\Psi(x,y,z)=(x,y,\kappa(\kappa(z)))$. In the case of $(X,r)$ of Example \ref{E:S4}, $\kappa$ can be either $\sigma=(12)$ or $\gamma=(34)$, giving the twisted YBE solutions $(x,y)\mapsto (y, \sigma(\gamma (x))) $ and $(x,y)\mapsto (\sigma(\gamma (y)), x) $, respectively. 
\end{ex}

Now we briefly recall the definition of a group with braiding operator from \cite{LYZ}. A \emph{braiding operator} on a group $(G,m,e)$ is a map $r:G^{2} \rightarrow G^{2}$ satisfying 
\begin{align}
r(e,g)=(g,e), \quad  r(g&,e) =(e,g) \label{Eq:brd1}
\\rm_{12} = m_{23}&r_{12}r_{23}\label{Eq:brdOpr1}
\\rm_{23}=m_{12} &r_{23}r_{12}\label{Eq:brdOpr2}
\\ mr= m\label{Eq:brdcomm}
\end{align}
It follows by definition that $r$ is a non-degenerate solution of the YBE, Corollary 3 of \cite{LYZ}.
\begin{defi}\label{D:TwistSkw} A \emph{Drinfeld twist} on a group $(G,m,e)$ with a braiding operator $r$ consists of a Drinfeld twist $(F,\Phi,\Psi)$ on the underlying YBE solution $(G,r)$ satisfying the additional conditions 
\begin{align}
 \Psi ( x,y,e) = (x,y,e),  &\quad \Phi (e,x,y)= (e,x,y) \tag{G1}\label{Eq:G1}
\\ F(e,x) =(e,x), &\quad F(x,e)= (x,e) \tag{G2}\label{Eq:G2}
\\m_{23}\Phi(x,y,z) &=F(x,yz) \tag{G3} \label{Eq:G3}
\\m_{12}\Psi (x,y,z) &= F(xy,z) \tag{G4} \label{Eq:G4}
\end{align}
where $x,y,z\in G$. 
\end{defi}
Although this definition does not appear to be related to Drinfeld twists on Hopf algebras, the relation between the two will be clarified in Section \ref{S:SL}, by the study of Hopf algebras in $\SL$. 
\begin{lemma}\label{L:TwistSkw} If $(G,m,e)$ is a group with a braiding operator $r$, and $(F,\Phi,\Psi)$ is as in Definition \ref{D:TwistSkw}, it follows that 
\begin{align}
\Phi (x,y,e)=(F(x,y),e), \quad \Phi (x, e, y)= F_{13} (x,e,y)\label{Eq:L1}\tag{L1}
\\ \Psi (e,x,y) = (e,F(x,y)), \quad \Psi (x, e, y)= F_{13}(x,e,y) \label{Eq:L2}\tag{L2}
\end{align}
hold for any $x,y\in G$.
\end{lemma}
\begin{proof} The second equation in \eqref{Eq:L1} follows from the first equation together with \eqref{Eq:T2} and \eqref{Eq:brd1} and the first equation in \eqref{Eq:L1} is a consequence of \eqref{Eq:T1}, \eqref{Eq:G1} and \eqref{Eq:G2}: 
\begin{align*}
\Phi (x,y,e)&= F_{23}^{-1}F_{23}\Phi (x,y,e)= F_{23}^{-1}F_{12}\Psi (x,y,e)= F_{23}^{-1}\big( F(x,y),e\big)= \big(F(x,y),e\big)
\\ \Phi (x,e,y)&= r_{23}^{-1}\Phi r_{23}(x,e,y)= r_{23}^{-1}\Phi (x,y,e)= r_{23}^{-1} \big(F(x,y),e\big)= F_{13} (x,e,y)
\end{align*}
The proof of \eqref{Eq:L2} is completely analogues. 
\end{proof}
Observe that by Lemma \ref{L:TwistSkw}, $F$ can be recovered from $\Phi$ or $\Psi$, and twists as defined in Definition \ref{D:TwistSkw} are completely determined by the pair $(\Phi,\Psi)$. Additionally, \eqref{Eq:T1} implies that $\Psi$ is determined by $F$ and $\Phi$. Hence, Definition \ref{D:TwistSkw} can then be re-written in terms of a single map $\Phi$: a twist on a group $(G,m,e)$ with a braiding operator $r$ consists of a bijection $\Phi: X^{3}\rightarrow X^{3}$ such that $\Phi (x,y,e)= (\ov{\Phi}(x,y),e) $ for a bijection $\ov{\Phi}:X^{2}\rightarrow X^{2}$, and $\Phi$ satisfies \eqref{Eq:T2} and 
\begin{align}
\Phi (e, x,y) = (e,x,&y), \quad \ov{\Phi}(x,e)=(x,e)\label{Eq:Z1}
\\m_{23}\Phi &=\ov{\Phi}m_{23}\label{Eq:Z2}
\\ m_{12}\ov{\Phi}_{12}^{-1}\ov{\Phi}_{23}\Phi &=\ov{\Phi} m_{12} \label{Eq:Z3}
\\ r_{12}\ov{\Phi}_{12}^{-1}\ov{\Phi}_{23}\Phi&= \ov{\Phi}_{12}^{-1}\ov{\Phi}_{23}\Phi r_{12} \label{Eq:Z4}
\end{align}
for $x,y \in G$. Notice that $\Psi$ is recovered from $\Phi$ as $\ov{\Phi}_{12}^{-1}\ov{\Phi}_{23}\Phi $, so that \eqref{Eq:T1} holds and \eqref{Eq:Z3} and \eqref{Eq:Z4} are translations of \eqref{Eq:G4} and \eqref{Eq:G3}, respectively. Similarly, \eqref{Eq:Z2} replaces \eqref{Eq:G3}. The remaining conditions in \eqref{Eq:G1} and \eqref{Eq:G2} follow by \eqref{Eq:Z1} and \eqref{Eq:Z2}, since $\Phi (e,x,e)= (\ov{\Phi}(e,x),e)$ and $\ov{\Phi}_{12}^{-1}\ov{\Phi}_{23}\Phi(x,y,e)= \ov{\Phi}_{12}^{-1}\ov{\Phi}_{23}(\ov{\Phi}(x,y),e)= \ov{\Phi}_{12}^{-1}\ov{\Phi}_{12}(x,y,e)$. Although, the data for a twist can be captured in a single map $\Phi$, we will continue to use triples $(F,\Phi,\Psi)$, since they are much more intuitive to work with and make our proofs less complicated. 

\begin{thm}\label{T:TwistSkw} Let $(F,\Phi,\Psi)$ be a Drinfeld twist on a a group $(G,m,e)$ with a braiding operator $r$. Then $(G,mF^{-1}, e)$ defines a new group structure on the set $G$ with a braiding operator $FrF^{-1}$. 
\end{thm}
\begin{proof} We must first demonstrate that $(G,mF^{-1}, e)$ is a group. First, we observe that $mF^{-1}$ is an associative operation by \eqref{Eq:T1} and \eqref{Eq:G3}, \eqref{Eq:G4}:
\begin{align*}
(mF^{-1})(mF^{-1})_{12}&= mF^{-1}m_{12}F^{-1}_{12}= m m_{12}\Psi^{-1}F^{-1}_{12}= m m_{23}\Phi^{-1}F_{23}^{-1}
\\&= mF^{-1}m_{23}F_{23}^{-1}= (mF^{-1})(mF^{-1})_{23}
\end{align*}
By definition \eqref{Eq:G2} implies that $e\in G$ acts as the identity element for $mF^{-1}$. What remains is the existence of inverses with respect to $mF^{-1}$. Let $x\in G$, then for any element $y\in G$, we first observe that $\Phi (y,x,x^{-1})= (y, z,z^{-1})$, for some $z\in G$, since 
\begin{align*}
m_{23}\Phi (y,x,x^{-1})= F m_{23} (y,x,x^{-1})= F(y,e)=(y,e)
\end{align*}
A similar argument shows that $\Psi (x,x^{-1},y)= ( l,l^{-1},y)$, for some $l\in G$. Now we observe that $F_{23}\Phi (x,x^{-1},x)= (x,x_{F},x^{-1})$, for some $x_{F}$. By the previous argument, there exists a pair $z,l\in G$ such that
\begin{align*}
\big( x,F(z,z^{-1}) \big)= F_{23}\Phi (x,x^{-1},x)&= F_{12}\Psi (x,x^{-1},x)= \big(F(l,l^{-1}),x\big) 
\\ \Rightarrow mF^{-1}(x_{F}, x)= m(z,z^{-1})=e,&\quad  mF^{-1}(x,x_{F})=m(l,l^{-1})=e
\end{align*}
Consequently, $x_{F}$ is the inverse of $x$ with respect to $mF^{-1}$ and $(G,m,1)$ forms a group.

In the second part of the proof we must show that $FrF^{-1}$ is a braiding operator on the group. Condition \eqref{Eq:brdcomm} follows by definition since $(mF^{-1})(FrF^{-1})= mrF^{-1}=mF^{-1}$ and \eqref{Eq:brd1} follows by \eqref{Eq:G2} and $r$ satisfying \eqref{Eq:brd1} e.g. $FrF^{-1}(1,x)= Fr(1,x)=F(x,1)=(x,1)$. To prove \eqref{Eq:brdOpr1}, we utilise the equivalent form of \eqref{Eq:T1}, $F_{23}^{-1}F_{12}= \Phi\Psi^{-1}$, and \eqref{Eq:T2}, \eqref{Eq:G3}, \eqref{Eq:T1} and \eqref{Eq:G4}, respectively:
\begin{align*}
 (mF^{-1})_{23}&(FrF^{-1})_{12}(FrF^{-1})_{23}= m_{23}F^{-1}_{23}F_{12}r_{12}F^{-1}_{12}F_{23}r_{23}F^{-1}_{23}=m_{23}\Phi \Psi^{-1}r_{12}\Psi \Phi^{-1} r_{23}F^{-1}_{23} 
 \\&=\Phi m_{23}r_{12} r_{23}\Phi^{-1}F^{-1}_{23} = F r m_{12}\Psi^{-1}F^{-1}_{12}= FrF^{-1}m_{12}F^{-1}_{12}=(FrF^{-1})(mF^{-1})_{12}
\end{align*}
Proving \eqref{Eq:brdOpr2} is analogues to the argument above and is left to the reader.
\end{proof}

The above results have been stated in terms of groups with braiding operators, which are equivalent structures to skew braces. The latter objects consist of a set $B$ with two compatible group structures $(B,.)$ and $(B,\star )$. In particular, any group $(G,.)$ with braiding operator $r$ gives rise to a skew braces $(G,., \star)$, where $x\star y:= x.\sigma_{x}^{-1}(y)$, while the braiding operator $r$ can be recovered from the group operations of a skew brace $(G,.,\star)$. We refer to Section 3 of \cite{smoktunowicz2018skew} for additional details, where the products $.$ and $\star$ are denoted by $\circ$ and $.$ , respectively. Although $r$ can be determined in terms of the skew brace operations, $.$ and $\star$, the author has not been able to find a way to re-write the axioms of Definition \ref{D:TwistSkw} solely in terms of $.$ and $\star$. However, by the end of this section we will fully classify all twists between skew braces and describe them in terms of the skew brace operations. From this point onwards, we will use the terms group with braiding operator and skew brace interchangeably and refer to $(G,\star )$ as the \emph{additive group} of the skew brace $(G,.,\star)$, as done in \cite{smoktunowicz2018skew}.

In the next example we will reflect on the difference between twists on a YBE solution $(X,r)$ and twists on the solution's universal skew brace $G(X,r)=\langle x\in X \mid x.y= \sigma_{x}(y).\gamma_{y}(x), \ \forall x,y\in X\rangle $. 
\begin{ex}\label{E:Universal} Let $(X,r)$ be as in Example \ref{E:S4}. As discussed in Example 3.4 of \cite{smoktunowicz2018skew}, the universal skew brace of $(X,r)$, denoted by $G$, will be the free abelian group on two generators, $a$ and $b$, with the flip map $\fr{fl}_{G}$ as its braiding operator. The natural map $i: X\rightarrow G$ sends $1,2$ to generator $a$ and $3,4$ to generator $b$. First, let us consider the problem of extending a twist $(f,\phi,\psi)$ on $(X,r)$ to a twist of skew braces $(F,\Phi,\Psi)$ on $(G,\fr{fl}_{G})$, so that the triples commute with $i$. The issue here is that $i$ is not injective. For example the triple $f= (13)\times (13)$, $\phi =(13)\times\id_{X^{2}}$, $\psi=\id_{X^{2}}\times (13)$ define a twist on $(X,r)$, while there cannot exist a map $F$ such that $F(i\times i) = (i\times i )f$ holds true since $i(1)=i(2)\neq i(3)$. Even if we could define the triple on the generators, there is no unique way to extend maps such as $\psi$ and $\phi$ to multiplications of the generators e.g. $\Phi (a^{3},a^{2},a^{7})$ is not uniquely determined by $(F,\Phi,\Psi)$ being defined on the generators of the group. The reader should also note that the twists in Examples \ref{E:S4} and \ref{E:noninvLyuba} all lift to the trivial twist where $F,\Phi$ and $\Psi$ are all identity morphisms. In the opposite direction, it is also non-trivial to check whether a twist on $G$ extends a canonical twist on $X$. The issue here is that a generator in $G$ represents several elements of $X$. Consider the group automorphisms $\chi_{a^{n}b^{m}}: G\rightarrow G $ corresponding to group elements $a^{n}b^{m}\in G$ defined by 
\begin{equation*}
\chi_{a^{n}b^{m}}(a^{l}b^{k})=\begin{cases}
a^{l}b^{k}\text{ if }n+m\text{ is even}
\\b^{l}a^{k}\text{ if }n+m\text{ is odd}
\end{cases} 
\end{equation*}
We obtain a twist on $(G,\fr{fl}_{G})$ defined by $F(g,h)=(g,\chi_{g} (h))$ and $\Phi (g,h,k)= (g,\chi_{g} (h),\chi_{g} (k))$ and $\Psi=(g,h,\chi_{gh}(k)) $, which recovers the universal skew brace appearing in Example 3.12 of \cite{bachiller2018solutions}. In this case, we can define $f(x,y)= (\varpi ( x), \varphi(y))$ with any of the choices $\varpi=\id_{X},(12),(34),(12)(34)$ and $\varphi=(13)(24),(14)(23)$ and $F(i\times i) = (i\times i )f$ will follow regardless of our choice.  Consequently, there is no natural choice for a triple $(f,\phi,\psi)$ so that $(F,\Phi,\Psi)$ is a lift of the former. For example, the triples $f= \varpi \times (14)(23)$, $\phi= \varpi \times (13)(24)\times (14)(23)$ and $\psi = \varpi \times \varpi\times \id_{X}$ form twists lifting to $(F,\Phi,\Psi)$ for both choices $\varpi =(12)$ and $\varpi =(34)$. 
\end{ex} 

\begin{thm}\label{T:SkwGroupoid} The category which has skew braces as objects and Drinfeld twists between them as morphisms forms a subgroupoid of the groupoid of YBE solutions of Theorem \ref{T:YBEGroupoid}. 
\end{thm} 
\begin{proof} We only need to show that Drinfeld twists on skew braces respect the composition and inverse operation defined in Theorem \ref{T:YBEGroupoid}. Recall from the proof of Theorem \ref{T:YBEGroupoid}, the inverse of a triple $(F,\Phi,\Psi)$ is given by $(F^{-1}, F_{23}\Phi^{-1} F^{-1}_{23},F_{12}\Psi^{-1} F^{-1}_{12})$ and we must demonstrate that it is indeed a twist on $(G,mF^{-1},e)$ with braiding operator $FrF^{-1}$. If the first triple satisfies \eqref{Eq:G1} and \eqref{Eq:G2}, it follows directly that the inverse triple also satisfies these axioms, since the unit element of the twisted skew brace remains the same. Now we demonstrate \eqref{Eq:G3}:
\begin{align*}
(mF^{-1})_{23}F_{23}\Phi^{-1} F^{-1}_{23}= m_{23}F^{-1}_{23}F_{23}\Phi^{-1} F^{-1}_{23}=  m_{23}\Phi^{-1} F^{-1}_{23}= F^{-1}m_{23} F^{-1}_{23}= F^{-1}(m F^{-1})_{23}
\end{align*}
Condition \eqref{Eq:G4} holds in a symmetric manner. 

We must also check that composing twists as defined in Theorem \ref{T:YBEGroupoid} respect the additional axioms in Definition \ref{D:TwistSkw}. Let $(F,\Phi,\Psi)$ be a twist on $(G,m,e,r)$ and $(\mathbf{F},\phi,\psi)$ be a twist on $(G,mF^{-1}, e, FrF^{-1})$. Then the composite twist $(\mathbf{F}F, F_{23}^{-1}\phi F_{23}\Phi, F_{12}^{-1}\psi F_{12}\Psi)$ satisfies \eqref{Eq:G1} and \eqref{Eq:G2} by triviality. Conditions \eqref{Eq:G3} holds since since $(F,\Phi)$ and $(\mathbf{F},\phi)$ satisfy \eqref{Eq:G3}:
\begin{align*}
m_{23}F_{23}^{-1}\phi F_{23}\Phi= (mF^{-1})_{23}\phi F_{23}\Phi= \mathbf{F}(mF^{-1})_{23}F_{23}\Phi= \mathbf{F}m_{23}\Phi=\mathbf{F} F m_{23}
\end{align*}
Condition \eqref{Eq:G4} holds in a symmetric manner and is left to the reader. 
\end{proof}

Note that Drinfled twists on YBE solutions and skew braces can be defined upto isomorphism. Explicitly, if $(G,m,.)$ with operator $r$ and $(H,m',e')$ with $r'$ are isomorphic via a bijection $f:G\rightarrow H$ satisfying $m' (f\times f)= fm$ and $(f\times f)r=r'(f\times f)$, then twists $(F,\Phi,\Psi)$ on $(G,m,e,r)$ are in bijection with twists $(F',\Phi',\Psi')$ on $(H,m',e',r')$, where the latter pairs are obtained by conjugation by $f\times f$ and $f\times f\times f$ e.g. $F':= (f\times f)F(f^{-1}\times f^{-1})$. Additionally, the twisted skew braces $(G,mF^{-1},e, FrF^{-1})$ and $(G,m(F')^{-1},e, F'r'(F')^{-1})$ obtained from corresponding twists between a pair of isomorphic skew braces will be isomorphic via the same map $f:X\rightarrow Y$, since $(f\times f)FrF^{-1}=F'r'(F')^{-1}(f\times f)$ and $ fmF^{-1}=m'(F')^{-1}(f\times f)$. Hence, we could refine the groupoids of Theorems \ref{T:YBEGroupoid} and \ref{T:SkwGroupoid} to have isomorphism classes of YBE solutions and skew braces as objects, respectively. 

\begin{ex}\label{E:Theta} We can realise the twists defined in Example \ref{E:Doikou} for any skew brace $(G,.,r)$. Define $F(x,y)=(x,\sigma_{x}(y))$, $\Phi (x,y,z)= (x,\sigma_{x}(y), \sigma_{\gamma_{y}(x)}(z))$ and $\Psi(x,y,z)=(x,y,\sigma_{x}(\sigma_{y}(z)))$, then $(F,\Phi, \Psi)$ form a twist on the underlying solution $(G,r)$ by Example \ref{E:Doikou} and \eqref{Eq:G1} and \eqref{Eq:G2} hold trivially by \eqref{Eq:brd1}, while \eqref{Eq:G3} and \eqref{Eq:G4} follow from \eqref{Eq:brdOpr2} and \eqref{Eq:brdOpr1}, respectively. By Theorem \ref{T:TwistSkw} we obtain a new group structure on $G$ defined by $mF^{-1}(x,y)=x.\sigma_{x}^{-1}(y)$ which agrees with the additive group structure of the skew brace, $\star$, and a corresponding braiding operator $FrF^{-1}(x,y)= (y, \sigma_{y}(\gamma_{\sigma^{-1}_{x} (y)}(x)))$. Notice that $\sigma_{y}(\gamma_{\sigma^{-1}_{x} (y)}(x))= y^{\star} \star x \star y$, where $y^{\star}$ denotes the inverse of $y\in G$ with respect to the operation $\star$. In the terminology of \cite{guarnieri2017skew}, $F(x,y)= (x,x^{-1}.(x\star y))$ and this twists sends a skew brace $(G,.,\star)$ to the trivial skew brace  $(G,\star, \star)$. The latter is the skew brace obtained by a single group $(G,\star)$ as described in Example 1.3 of \cite{guarnieri2017skew}.
\end{ex}

By Example \ref{E:Theta}, it follows that the any skew brace $(G,.,\star)$ is related to the trivial skew brace $(G,\star,\star)$ by a twist. Since Drinfeld twists can be inverted and composed, by Theorem \ref{T:SkwGroupoid}, we conclude that any two skew braces with the same additive groups are related via Drinfeld twists. Two observations follow from this result. Firstly, it is well known that there exist non-isomorphic skew braces with isomorphic additive groups. Hence, the notion of a twist relating two skew braces is much weaker than isomorphism of skew braces. Secondly, the problem of whether two skew braces $(G,.,\star )$ and $(G,.',\star')$ are related by Drinfeld twists fully translates into a problem about their underlying additive groups and whether there exists Drinfeld twists between $(G,\star,\star)$ and $(G,\star',\star')$, which we can answer. 
\begin{thm}\label{T:GrpTheoretic} If two skew braces $(G,\circ ,.)$ and $(G,.',\star)$ are related via a Drinfeld twist, then their respective additive groups $(G,.)$ and $(G,\star)$ are isomorphic. 
\end{thm}
\begin{proof} One direction of the argument has already been proved in the discussion before Theorem \ref{T:GrpTheoretic}. For the other direction, let us assume that there exists a Drinfeld $(G,\circ ,.)$ relating $(G,.',\star)$. Then by Theorem \ref{T:SkwGroupoid} and Example \ref{E:Theta}, we can obtain a Drinfeld twist relating $(G,.,.)$ and $(G,\star,\star)$. Hence, let $(F,\Phi,\Psi)$ be a twist on $(G,.,.)$ such that $\star =.F^{-1}$ and for any $x,y\in G$ we have $r'(x,y):=FrF^{-1}(x,y)=  (y,y^{\star}\star x \star y )$, where $r(x,y)= (y,y^{-1}xy)$. Since $\star F=.$, we can write $F(x,y)= \big( f_{xy}(x),(f_{xy}(x))^{\star}\star (x. y) \big)$ for a family of maps $f_{g}:G\rightarrow G$ corresponding to $g\in G$. Since $F$ is a bijection, then the maps $f_{g}$ are also bijective. Consequently, we have that 
\begin{align*}
 Fr(x,y ) = \big( f_{xy}(y),(f_{xy}(y))^{\star}\star (x y) \big)= \big((f_{xy}(x))^{\star}\star (x. y),(xy)^{\star}\star f_{xy}(x) \star (xy) \big)=r'F(x,y)
\end{align*}
which implies that $f_{xy}(x)\star f_{xy}(y)= xy$ for any pair $x,y\in G$ and $F(x,y)=( f_{xy}(x),f_{xy}(y))$.

By \eqref{Eq:G3} it follows that $\Phi(x,y,z)= (f_{xyz}(x),f_{xyz}(yz) (\alpha_{x,yz}(z))^{-1}, \alpha_{x,yz}(z))$ for some maps $\alpha_{x,c}:G\rightarrow G$ corresponding to $x,c\in G$. Expanding equation \eqref{Eq:T2} gives
\begin{align*}
\big( f_{xyz}(x), f_{xyz}(yz) (\alpha_{x,yz}(z^{-1}yz))^{-1}, \alpha_{x,yz}(z^{-1}yz)\big)= \big(f_{xyz}(x),  \alpha_{x,yz}(z), (\alpha_{x,yz}(z))^{-1}f_{xyz}(yz)\big)
\end{align*}
which implies that $f_{xc}(c)= \alpha_{x,c}(z)\alpha_{x,c}(z^{-1}c)$ for any $x,c,z\in G$, where we have replaced the term $yz$ by $c$ since $y$ is a free variable. We can also re-write $\Phi(x,y,z)= (f_{xyz}(x),\alpha_{x,yz}(y), \alpha_{x,yz}(z))$. A symmetric argument shows that $\Psi (x,y,z)= (\beta_{xy,z}(x),\beta_{xy,z}(y) ,f_{xyz}(z))$ for some maps $\beta_{c,z}:G\rightarrow G$ satisfying $\beta_{xy,z}(x).\beta_{xy,z}(y)= f_{xyz}(xy)$.

Now we compare the third component of equation \eqref{Eq:T1}. The third component of $F_{12}\Psi(x,y,z)$ is $f_{xyz}(z)$, while the third component of $F_{23}\Phi(x,y,z)$ is $f_{f_{xyz}(yz)}(\alpha_{x,yz}(z))$. Equating the two values tells us that $\alpha_{x,c}= f^{-1}_{f_{xc}(c)}f_{xc}$ and the identity $f_{xc}(c)= \alpha_{x,c}(z).\alpha_{x,c}(z^{-1}c)$ can be re-written as 
\begin{align*}
f_{xc}(c)=f^{-1}_{f_{xc}(c)}(f_{xc}(z)). f^{-1}_{f_{xc}(c)}(f_{xc}(z^{-1}c))
\end{align*}
where $x,c,z\in G$ are all independent variables. By $f_{xy}(x)\star f_{xy}(y)= xy$, we have that
\begin{align*}
f_{xc}(z)\star f_{xc}(z^{-1}c)= f_{f_{xc}(c)}\left(f^{-1}_{f_{xc}(c)}(f_{xc}(z))\right)\star f_{f_{xc}(c)}\left(f^{-1}_{f_{xc}(c)}(f_{xc}(z^{-1}c))\right)= f_{xc}(c)
\end{align*}
Thereby, we can change variables by $x=pc^{-1}$ and $c=zz'$ and observe that $f_{p}(z)\star f_{p}(z') = f_{p}(z.z')$ for all $p,z,z'\in G$. Hence, $f_{p}: (G,.)\rightarrow (G,\star)$ are groups isomorphisms for all $p\in G$.  
\end{proof}
We could continue the proof of Theorem \ref{T:GrpTheoretic} by comparing the first component of equation \eqref{Eq:T1}, which implies that $\beta_{x,y}= f^{-1}_{f_{xy}(x)}f_{xy}$ and the identity on $\beta_{xy,z}$ follows by $f_{p}$ being group morphisms. With this, \eqref{Eq:T1} holds automatically and \eqref{Eq:G1} and \eqref{Eq:G2} follow by $f_{g}(1)=1$. Hence, this result fully classifies all possible twists from $(G,.,.)$ to $(G,\star , \star)$, for isomorphic $(G,.)$ and $(G,\star)$. Given any family of group isomorphisms $\lbrace f_{x}: (G,.)\rightarrow (G,\star)\rbrace_{x\in G}$ satisfying $f_{x}(x)=x$, we obtain a non-trivial twist on $(G,.,.)$ by setting $F(x,y)=(f_{xy}(x),f_{xy}(y))$ and $\Phi$ and $\Psi$ as defined in the proof of Theorem \ref{T:GrpTheoretic}, with $\alpha_{x,y}=f^{-1}_{f_{xy}(y)}f_{xy}$ and $\beta_{x,y}= f^{-1}_{f_{xy}(x)}f_{xy}$. In the groupoid of skew braces and twists of Theorem \ref{T:SkwGroupoid}, these twists are all possible morphisms from $(G,.,.)$ to $(G,\star,\star)$. In particular, all endomorphisms of $(G,.,.)$ in this groupoid must also come from a family of group automorphisms $\lbrace f_{x}\rbrace_{x\in G}$ satisfying $f_{x}(x)=x$. 

Recall that in an arbitrary groupoid, if we are given morphisms $a: o_{1}\rightarrow o_{2}$, $a_{i}:o_{i}\rightarrow o'_{i}$ for $i=1,2$ between some objects $o_{1},o_{2},o'_{1},o'_{2}$, then there must exist a morphism $b: o'_{1}\rightarrow o'_{2}$ defined by $b=a_{2}aa_{1}^{-1}$ such that $a= a_{2}^{-1}ba_{1}$. Hence, we observe that any twist $(F,\Phi, \Psi):(G,.,\star) \rightarrow (G,.',\star')$ can be decomposed as the composition of twists coming from Example \ref{E:Theta} and a twist from $(G,\star,\star)$ to $(G,\star' ,\star')$ which by the above result has to come from a family of group isomorphisms.

\begin{corollary}\label{C:Decomposition} Any twist $(F,\Phi, \Psi):(G,.,\star) \rightarrow (G,\circ,\star)$ can be decomposed as the composition of the twists in Example \ref{E:Theta} and those coming from Theorem \ref{T:GrpTheoretic}. In particular, there exists a family of group isomorphisms $\lbrace f_{x}: (G,\star)\rightarrow (G,\star')\rbrace_{x\in G}$ satisfying $f_{x}(x)=x$ so that
\begin{equation}\label{Eq:AnyTwist}
F(x,y)=\left( f_{x.y}(x), \underline{\sigma}_{f_{x.y}(x)}^{-1}\big(f_{x.y}(\sigma_{x}(y))\big)\right) =\left( f_{x.y}(x),f_{x.y}(x)^{\circ} \circ (x.y)\right)
\end{equation}
where $\sigma$ and $ \underline{\sigma}$ denote the relevant maps for the braiding operators of $(G,.,\star)$ and $(G,\circ,\star)$, respectively. 
\end{corollary}
\section{Co-twists on Hopf Algebras in $\SL$}\label{S:SL}
In this section we adapt the notation of \cite{ghobadi2020skew} where we described skew braces as remnants of CQHAs in $\SL$ and study Drinfeld co-twists on these Hopf algebras. In particular, we look at co-twists on the $\SL$-FRT algebras, which were constructed in Section 5 of \cite{ghobadi2020skew}. The reader should firstly note that the Hopf algebras we will describe are not linear and are Hopf algebra objects in the category $\SL$ and secondly that there are two FRT-type algebras which we are concerned with: $H_{\omega}$ denotes the $\SL$-FRT algebra of a non-degenerate YBE solution, which recovers the universal skew brace of the solution as its remnant, while $H_{\omega_{m}}$ denotes the FRT-type algebra defined for a given skew brace, which recovers the skew brace itself as its remnant. From this point onwards, we will assume that all YBE solutions of concern are non-degenerate.

A \emph{Drinfeld co-twist} on a Hopf algebra $\ct{H}$ in $\SL$ consists of a morphism $\ct{F}: \ct{H}\tn \ct{H} \rightarrow \Pl{1} $, for which there exists another morphism $\ct{F}^{-1} :\ct{H}\tn\ct{H}\rightarrow \Pl{1}$ so that
\begin{align}
\ct{F}^{-1} (a_{(1)},b_{(1)}). \ct{F} (a_{(2)},b_{(2)}) =\epsilon (a).\epsilon &(b)= \ct{F} (a_{(1)},b_{(1)})  . \ct{F}^{-1}(a_{(2)},b_{(2)})\tag{cT1} \label{Eq:ConvInv}
\\\ct{F}\left( a_{(1)}. b_{(1)},c\right).\ct{F}\left( a_{(2)},b_{(2)}\right)=&\ \ct{F}\left( a,b_{(1)}
.c_{(1)}\right).\ct{F}\left(b_{(2)}, c_{(2)}\right)\tag{cT2} \label{Eq:Twisthopf}
\\ \ct{F} (a,1) =\epsilon (a)&= F(1,a)\tag{cT3}\label{Eq:Twistcounit}
\end{align}
hold for $a,b\in \ct{H}$. The first condition \eqref{Eq:ConvInv} says that $\ct{F}$ is \emph{convolution invertible} and \eqref{Eq:Twisthopf} and \eqref{Eq:Twistcounit} are sometimes called the cocycle conditions, see Section 2.3 of \cite{majid2000foundations} for more details. Given a co-twist on a co-quasitriangular Hopf algebras $(\ct{H},\Rr)$, we obtain a new CQHA $(\ct{H}^{\ct{F}},\Rr^{\ct{F}})$, where the multiplication and co-quasitriangular structure are twisted:
\begin{align}
m^{\ct{F}} (a,b) &= \ct{F}^{-1}\left(a_{(1)},b_{(1)} \right).a_{(2)}.b_{(2)}.\ct{F}\left( a_{(1)},b_{(1)} \right) \label{Eq:MultiTwis}
\\ \Rr^{\ct{F}} (a,b) &= \ct{F}^{-1}\left(a_{(1)},b_{(1)} \right)\Rr ( a_{(2)},b_{(2)}).\ct{F}\left(b_{(3)}, a_{(3)} \right) \label{Eq:RmatTwis}
\end{align}
where $a,b\in \ct{H}$. We refer the reader to Section 2.3 of \cite{majid2000foundations}, where the proof of the dual of this statement is proved. As mentioned in Chapter 9 of \cite{majid2000foundations}, these proofs can be re-written in diagrammatic manner and dualised. 

Recall from Section 3 and 4 of \cite{ghobadi2020skew} that given a Hopf algebra $\ct{H}$ in $\SL$, we can construct a group called its remnant, $R(\ct{H})$, and any co-quasitriangular structure on $\ct{H}$ induces a braiding operator on $R(\ct{H})$. The underlying set of $R(\ct{H})$ is defined solely by the counit of the Hopf algebra $\ct{H}$. Hence, given a twist $\ct{F}$ on a CQHA, $(\ct{H},\ct{R})$, the new skew brace, $R(\ct{H}^{\ct{F}})$, which we obtain will have the same underlying set as $R(\ct{H})$. Next we will show that $\ct{F}$ in fact induces a Drinfeld twist on $R(\ct{H})$ in the sense of Section \ref{S:DriTwis}. 

Before proving this statement, let us briefly emphasize the notation used in \cite{ghobadi2020skew}. Given a CQHA $(\ct{H},\Rr)$, there is a natural quotient Hopf algebra $\ct{Q}$ with a Hopf algebra projection $\pi:\ct{H}\rightarrow \ct{Q}$. The inclusion morphism $\iota: \ct{Q}\rightarrow \ct{H}$ sends an element $\pi (a)=\ov{a}\in \ct{Q}$ to $a\vee D$, where $D= \vee \epsilon^{-1}(\emptyset)$. Lastly, the remnant group of $\ct{H}$ is identified as the basis of $\ct{Q}$ i.e. $\ct{Q}=\Pl{R(\ct{H})}$. For more details we refer the reader to Section 3 of \cite{ghobadi2020skew}.
\begin{thm}\label{T:AnyCQHA} Any co-twist $\ct{F}$ on a co-quasitriangular Hopf algebra $(\ct{H},\Rr)$ induces a twist on its remnant skew brace $R(\ct{H})$. 
\end{thm}
\begin{proof} Given a co-twist $\ct{F}$, we will define the triple $(F,\Phi,\Psi)$ more generally for $\ct{Q}$ and then by demonstrating that they are invertible, it follows that they must restrict to maps on the basis elements. For $\ov{a},\ov{b},\ov{c}\in \ct{Q}$, define the triple as follows:
\begin{align*}
F(\ov{a},\ov{b})&= \ct{F}\left(a_{(1)},b_{(1)} \right).\big(\pi (a_{(2)}),\pi( b_{(2)})\big).\ct{F}^{-1}\left( a_{(3)},b_{(3)} \right)
\\ \Phi (\ov{a},\ov{b},\ov{c} )&= \ct{F}\left(a_{(1)},b_{(1)}.c_{(1)} \right).\big(\pi (a_{(2)}),\pi( b_{(2)}),\pi(c_{(2)})\big).\ct{F}^{-1}\left( a_{(3)},b_{(3)}.c_{(3)} \right)
\\ \Psi (\ov{a},\ov{b},\ov{c} )&= \ct{F}\left(a_{(1)}.b_{(1)},c_{(1)} \right).\big(\pi (a_{(2)}),\pi( b_{(2)}),\pi(c_{(2)})\big).\ct{F}^{-1}\left( a_{(3)}.b_{(3)},c_{(3)} \right)
\end{align*} 
Note that a more natural definition would have been of the form $F= (\pi\tn \pi) F' (\iota\tn\iota )$ for some $F':\ct{H}\tn\ct{H}\rightarrow \ct{H}\tn\ct{H}$. However, we can follow an argument made in the beginning of the proof of Theorem 4.2 in \cite{ghobadi2020skew}, since \eqref{Eq:ConvInv} implies that $F'$ commutes with $\epsilon\tn\epsilon$, and thereby composing with $\iota$ gives the same result.  

Now we define the inverse of $\Phi$ by 
\begin{align*}
\Phi^{-1} (\ov{a},\ov{b},\ov{c} )= \ct{F}^{-1}\left(a_{(1)},b_{(1)}.c_{(1)} \right).\big(\pi (a_{(2)}),\pi( b_{(2)}),\pi(c_{(2)})\big).\ct{F}\left( a_{(3)},b_{(3)}.c_{(3)} \right)
\end{align*}
The inverse of $F$ and $\Psi$ are obtained exactly in the same way, where $\ct{F}^{1}$ and $\ct{F}$ are swapped in their respective definitions. It then follows easily from \eqref{Eq:ConvInv}, that $\Phi^{-1}$ and $\Phi$ are inverses
\begin{align*}
\Phi^{-1}\Phi (\ov{a},\ov{b},\ov{c} ) =& \ct{F}\left(a_{(1)},b_{(1)}.c_{(1)} \right)\ct{F}^{-1}\left(a_{(2)(1)},b_{(2)(1)}.c_{(2)(1)} \right).\big(\pi (a_{(2)(2)}),\pi( b_{(2)(2)}),\pi(c_{(2)(2)})\big)
\\&.\ct{F}\left( a_{(2)(3)},b_{(2)(3)}.c_{(2)(3)} \right).\ct{F}^{-1}\left( a_{(3)},b_{(3)}.c_{(3)} \right)
\\=&  \ct{F}\left(a_{(1)},b_{(1)}.c_{(1)} \right)\ct{F}^{-1}\left(a_{(2)},b_{(2)}.c_{(2)} \right).\big(\pi (a_{(3)}),\pi( b_{(3)}),\pi(c_{(3)})\big).\ct{F}\left( a_{(4)},b_{(4)}.c_{(4)} \right)
\\&.\ct{F}^{-1}\left( a_{(5)},b_{(5)}.c_{(5)} \right)
\\=& \epsilon (a_{(1)}).\epsilon (b_{(1)}.c_{(1)}).\big(\pi (a_{(2)}),\pi( b_{(2)}),\pi(c_{(2)})\big)\epsilon (a_{(3)}).\epsilon (b_{(3)}.c_{(3)})= (\ov{a},\ov{b},\ov{c} )
\end{align*}
Similar proofs hold for $F$ and $\Psi$. Hence, $(F,\Phi,\Psi)$ restrict to automorphisms on the subsets of basis elements of $\ct{Q}^{\tn 2}$ and $\ct{Q}^{\tn 3}$ i.e. $R(\ct{H})^{2}$ and $R(\ct{H})^{3}$. Next we demonstrate that the axioms of Definition \ref{D:TwistSkw} hold. 

First, we recall that the unit of the group $R(\ct{H})$ is $\ov{1}$ and by \eqref{Eq:Twistcounit}, it follows directly that \eqref{Eq:G1} and \eqref{Eq:G2} hold. Additionally, recall that the multiplication on $R(\ct{H})$ is defined as the projection of multiplication of $\ct{H}$. Hence, if we denote the multiplication of $R(\ct{H})$ by $m$, then 
\begin{align*}
F(\ov{a},\ov{b}.\ov{c})&=F(\ov{a},\ov{b.c})= \ct{F}\left(a_{(1)},(b.c)_{(1)} \right).\big(\pi (a_{(2)}),\pi( (b.c)_{(2)})\big).\ct{F}^{-1}\left( a_{(3)},(b.c)_{(3)} \right)
\\&=\ct{F}\left(a_{(1)},b_{(1)}.c_{(1)} \right).\big(\pi (a_{(2)}),\pi( b_{(2)}.c_{(2)})\big).\ct{F}^{-1}\left( a_{(3)},b_{(3)}.c_{(3)} \right)
\\&=\ct{F}\left(a_{(1)},b_{(1)}.c_{(1)} \right).\big(\pi (a_{(2)}),\pi( b_{(2)}) .\pi (c_{(2)})\big).\ct{F}^{-1}\left( a_{(3)},b_{(3)}.c_{(3)} \right)= m_{23}\Phi (\ov{a},\ov{b}.\ov{c})
\end{align*}
holds. A symmetric argument show that \eqref{Eq:G4} also holds. 

It still remains to show that the axioms of Definition \ref{D:TwistYBE} hold i.e. we must demonstrate that $(F,\Phi,\Psi)$ defines a twist on the underlying YBE solution $(R(\ct{H}),r)$ as well. Observe that combining the fact that $\ct{F}$ is convolution invertible, \eqref{Eq:ConvInv}, and \eqref{Eq:Twisthopf} provides a symmetric relation for $\ct{F}^{-1}$: 
\begin{equation}\label{Eq:Twisthopf2}
\ct{F}^{-1}\left( a_{(1)},b_{(1)}\right).\ct{F}^{-1}\left( a_{(2)}. b_{(2)},c\right)= \ct{F}^{-1}\left(b_{(1)}, c_{(1)}\right) \ct{F}^{-1}\left( a,b_{(2)}
.c_{(2)}\right)\tag{cT2'}
\end{equation}
Condition \eqref{Eq:T1} follows directly from \eqref{Eq:Twisthopf} and \eqref{Eq:Twisthopf2}:
\begin{align*}
F_{23}\Phi  (\ov{a},\ov{b},\ov{c})=& \ct{F}\left(a_{(1)},b_{(1)}.c_{(1)} \right).\ct{F}\left( b_{(2)(1)},c_{(2)(1)} \right).\big(\pi (a_{(2)}),\pi( b_{(2)(2)}),\pi(c_{(2)(2)})\big)
\\ &.\ct{F}^{-1}\left(b_{(2)(3)},c_{(2)(3)} \right).\ct{F}^{-1}\left( a_{(3)},b_{(3)}.c_{(3)} \right)
\\=& \ct{F}\left(a_{(1)},b_{(1)}.c_{(1)} \right).\ct{F}\left( b_{(2)},c_{(2)} \right).\big(\pi (a_{(2)}),\pi( b_{(3)}),\pi(c_{(3)})\big).\ct{F}^{-1}\left(b_{(4)},c_{(4)} \right)
\\ &.\ct{F}^{-1}\left( a_{(3)},b_{(5)}.c_{(5)} \right)
\\ =&\ct{F}\left( a_{(1)}. b_{(1)},c_{(1)}\right).\ct{F}\left( a_{(2)},b_{(2)}\right)\big(\pi (a_{(3)}),\pi( b_{(3)}),\pi(c_{(2)})\big).\ct{F}^{-1}\left(a_{(4)},b_{(4)} \right)
\\ &.\ct{F}^{-1}\left( a_{(5)}.b_{(5)}.c_{(3)} \right)= F_{12}\Psi  (\ov{a},\ov{b},\ov{c})
\end{align*}
Recall from Section 4 of \cite{ghobadi2020skew} that $r$ is defined by 
\begin{equation}\label{Eq:BrdRem}
r ( \ov{a},\ov{b}) = \Rr\left(a_{(1)},b_{(1)} \right).\big(\pi( b_{(2)}),\pi (a_{(2)})\big).\Rr^{-1}\left( a_{(3)},b_{(3)} \right) 
\end{equation}
Moreover, $\Rr$ satisfies $\Rr (b_{(1)}, a_{(1)})a_{(2)}.b_{(2)}= b_{(1)}. a_{(1)} \Rr (b_{(2)} , a_{(2)})$ and as for $\ct{F}$ and \eqref{Eq:Twisthopf2}, we can obtain a symmetric relation for $\Rr^{-1}$, since $\Rr$ is convolution invertible. Using these relations we observe that \eqref{Eq:T2} holds: 
\begin{align*}
r_{23} \Phi (\ov{a},\ov{b},\ov{c})=&\ct{F}\left(a_{(1)},b_{(1)}.c_{(1)} \right).\Rr\left(b_{(2)(1)},c_{(2)(1)} \right).\big(\pi (a_{(2)}),\pi(c_{(2)(2)}),\pi( b_{(2)(2)})\big)
\\&.\Rr^{-1}\left( b_{(2)(3)},c_{(2)(3)} \right) .\ct{F}^{-1}\left( a_{(3)},b_{(3)}.c_{(3)} \right)
\\ =&\ct{F}\left(a_{(1)},b_{(1)}.c_{(1)} \right).\Rr\left(b_{(2)},c_{(2)} \right).\big(\pi (a_{(2)}),\pi(c_{(3)}),\pi( b_{(3)})\big).\Rr^{-1}\left( b_{(4)},c_{(4)} \right)
\\& .\ct{F}^{-1}\left( a_{(3)},b_{(5)}.c_{(5)} \right)
\\=&\ct{F}\left(a_{(1)},b_{(1)(1)}.c_{(1)(1)} \right).\Rr\left(b_{(1)(2)},c_{(1)(2)} \right).\big(\pi (a_{(2)}),\pi(c_{(2)}),\pi( b_{(2)})\big)
\\& .\Rr^{-1}\left( b_{(3)(1)},c_{(3)(1)} \right).\ct{F}^{-1}\left( a_{(3)},b_{(3)(2)}.c_{(3)(2)} \right)
\\=&\Rr\left(b_{(1)(1)},c_{(1)(1)} \right).\ct{F}\left(a_{(1)},c_{(1)(2)}.b_{(1)(2)}\right).\big(\pi (a_{(2)}),\pi(c_{(2)}),\pi( b_{(2)})\big)
\\& .\ct{F}^{-1}\left( a_{(3)},c_{(3)(1)}.b_{(3)(1)} \right).\Rr^{-1}\left( b_{(3)(2)},c_{(3)(2)} \right)
\\=&\Rr\left(b_{(1)},c_{(1)} \right).\ct{F}\left(a_{(1)},c_{(2)(1)}.b_{(2)(1)} \right).\big(\pi (a_{(2)}),\pi(c_{(2)(2)}),\pi( b_{(2)(2)})\big)
\\& .\ct{F}^{-1}\left( a_{(3)},c_{(2)(3)} .b_{(2)(3)}\right).\Rr^{-1}\left( b_{(3)},c_{(3)} \right)= \Phi r_{23} ( \ov{a},\ov{b},\ov{c})
\end{align*}
A symmetric argument shows that \eqref{Eq:T3} holds. 
\end{proof}

\begin{corollary}\label{C:HopfTwist} Let $\ct{F}$ be a a co-twist on $(\ct{H},\ct{R})$, and $(F,\Phi,\Psi)$ the induced twist on $R(\ct{H})$, by Theorem \ref{T:AnyCQHA}. The induced skew brace structure induced on $R(\ct{H}^{\ct{F}})$ agrees with that induced by the twist $(F,\Phi,\Psi)$.
\end{corollary}
\begin{proof} First recall that the underlying set of the remnant is solely dependent on the counit and thereby $G=R(H_{\omega_{m}})= R(H^{\ct{F}})$ as sets. Let us denote the twisted product on $R(\ct{H}^{\ct{F}})$ by $.^{\ct{F}}$, then by \eqref{Eq:MultiTwis} we have that 
\begin{align*}
\ov{a}.^{\ct{F}} \ov{b}=& \pi m^{\ct{F}} (a\vee D,b\vee D) = \ct{F}^{-1}\left(a_{(1)}\vee D_{(1)},b_{(1)}\vee D_{(1)}\right).\pi m\big(a_{(2)}\vee D_{(2)},b_{(2)}\vee D_{(2)}\big)
\\&.\ct{F}\left( a_{(3)}\vee D_{(3)},b_{(3)}\vee D_{(3)}\right)
\\= & \ct{F}^{-1}\left(a_{(1)},b_{(1)}\right).\pi m\big(a_{(2)}\vee D,b_{(2)}\vee D\big).\ct{F}\left( a_{(3)},b_{(3)}\right) = \pi m (\iota \tn \iota) F^{-1}(\ov{a},\ov{b})
\end{align*}
What remains to be shown is that $r^{\ct{F}}= FrF^{-1}$, where we denote the induced braiding operator on $R(H^{\ct{F}})$ by $r^{\ct{F}}$, which follows by \eqref{Eq:RmatTwis} and \eqref{Eq:BrdRem}: 
\begin{align*}
r^{\ct{F}}( \ov{a},\ov{b}) =& \ct{F}^{-1}\left(a_{(1)(1)},b_{(1)(1)} \right)\Rr ( a_{(1)(2)},b_{(1)(2)}).\ct{F}\left(b_{(1)(3)}, a_{(1)(3)} \right).\big(\pi( b_{(2)}),\pi (a_{(2)})\big)
\\&\ct{F}^{-1}\left(b_{(3)(1)},a_{(3)(1)} \right)\Rr^{-1}\left( a_{(3)(2)},b_{(3)(2)} \right) .\ct{F}\left( a_{(3)(3)},b_{(3)(3)} \right)
\\=&\ct{F}^{-1}\left(a_{(1)},b_{(1)} \right)\Rr ( a_{(2)},b_{(2)}).\ct{F}\left(b_{(3)(1)}, a_{(3)(1)} \right).\big(\pi( b_{(3)(2)}),\pi (a_{(3)(2)})\big)
\\&\ct{F}^{-1}\left(b_{(3)(3)},a_{(3)(3)} \right)\Rr^{-1}\left( a_{(4)},b_{(4)} \right) .\ct{F}\left( a_{(5)},b_{(5)} \right)
\\=&\ct{F}^{-1}\left(a_{(1)},b_{(1)} \right)\Rr ( a_{(2)},b_{(2)}).F \big(\pi( b_{(3)}),\pi (a_{(3)})\big).\Rr^{-1}\left( a_{(4)},b_{(4)} \right) .\ct{F}\left(a_{(5)}, b_{(5)}\right)
\\=&\ct{F}^{-1}\left(a_{(1)},b_{(1)} \right).Fr \big(\pi( a_{(2)}),\pi (b_{(2)})\big).\ct{F}\left(b_{(3)}, a_{(3)} \right)= FrF^{-1}( \ov{a},\ov{b}) \qedhere
\end{align*} 
\end{proof}
Each skew brace can arise as the remnant of non-isomorphic CQHAs in $\SL$. It is not clear whether every twist on a skew brace can be extended to co-twists on some CQHA which recovers it as its remnant. In what follows we will discuss co-twists on the two families of Hopf algebras in $\SL$ which were studied in \cite{ghobadi2020skew} and comment on the obstruction of extending a twist on the remnant to a co-twist on these algebras. 

In Section 5.1 of \cite{ghobadi2020skew}, we constructed the $\SL$-FRT algebra $H_{\omega}$ corresponding to a YBE solution $(X,r)$. As an algebra in $\SL$, $H_{\omega}$ is generated by elements $(x,y)_{1}$ and $(x,y)_{2}$, corresponding to pairs $x,y\in X$, and has relations
\begin{align} 
\vee_{a\in X} \st{(x,a)_{1}.(x,a)_{2}}=1= \vee_{a\in X}\st{(a,x)_{2}.(a,x)_{1}}\label{Eq:HwInv}
\\(x,a)_{1}.(y,a)_{2}= \emptyset=(a,x)_{2}.(a,y)_{1},\ \text{when }x\neq y\label{Eq:Hwzero}
\\(x,y)_{1}.(a,b)_{1}= (\sigma_{x}(a), \sigma_{y}(b) )_{1}.(\gamma_{a}(x), \gamma_{b}(y))_{1} \label{Eq:HwBrd}
\end{align}
for $a,b,x,y\in X$. We refer the reader to Section 5.1 of \cite{ghobadi2020skew} for additional details and the definitions of the structural maps of the Hopf algebra. It was also shown that the remnant of the $\SL$-FRT algebra of a solution recovers the universal skew brace of the solution. Consequently, by Theorem~\ref{T:AnyCQHA} co-twists on the $\SL$-FRT algebra $H_{\omega}$ of a solution induce twists on its universal skew brace. However, in \cite{ghobadi2020skew} we noted that $H_{\omega}$ carries much more information about the solution compared to the universal skew brace. This is confirm in the study of twists since any co-twist on $\SL$-FRT algebra of a solution also induces a twist on the YBE solution itself. 

\begin{thm}\label{T:TwistHw} If $H_{\omega}$ denotes the $\SL$-FRT algebra of a set-theoretical YBE solution $(X,r)$, then a Drinfeld co-twists on $H_{\omega}$ are induces a Drinfeld twists on $(X,r)$, as in Definition \ref{D:TwistYBE}. 
\end{thm}
\begin{proof} Let $\ct{F}:H_{\omega}\tn H_{\omega}\rightarrow \Pl{1}$ denote a co-twist on $H_{\omega}$. We define the triple $(F,\Phi,\Psi )$ as follows: 
\begin{align*}
F(x,y)= (a,b)&\text{ if and only if }\ct{F}\big(\st{ (x,a)_{1},(y,b)_{1}}\big)=1
\\ \Phi (x,y,z) =(a,b,c) &\text{ if and only if }\ct{F} \big( \st{(x,a)_{1},(y,b)_{1}. (z,c)_{1}}\big) =1
\\ \Psi (x,y,z) =(a,b,c) &\text{ if and only if }\ct{F} \big( \st{(x,a)_{1}.(y,b)_{1}, (z,c)_{1}}\big) =1
\end{align*}
We must first demonstrate that $F,\Phi$ and $\Psi$ are indeed maps, rather than relations between $X^{2}$ and itself. By \eqref{Eq:ConvInv}, we have that for a fixed $a,b,a',b'\in X$:
\begin{align*}
\bigvee_{m,n\in X}\ct{F}\big(\st{ (a,m)_{1},(b,n)_{1}}&\big). \ct{F}^{-1}\big(\st{ (m,a')_{1},(n,b')_{1}}\big)= \delta_{a,a'}.\delta_{b,b'}
\\ =&\bigvee_{m,n\in X}\ct{F}^{-1}\big(\st{ (a,m)_{1},(b,n)_{1}}\big). \ct{F}\big(\st{ (m,a')_{1},(n,b')_{1}}\big)
\end{align*}
Consequently, if we let $F^{-1}(x,y):= (a,b)$ if and only if $\ct{F}^{-1}\big(\st{ (x,a)_{1},(y,b)_{1}}\big)=1$, then by the above equation $FF^{-1}= F^{-1}F=\id_{X^{2}}$. Hence, both $F$ and $F^{-1}$ must be bijective maps, since a relation $\mathsf{R}\subset A\times B$ between two sets is invertible by composition of relations, if and only if it is represents a bijective map. Similarly, \eqref{Eq:ConvInv} implies that $\Phi$ and $\Psi$ as we defined them are both well-defined and bijective maps from $X^{3}$ to itself.

Now we substitute $a= (a,a')_{1} $, $b=(b,b')_{1}$ and $c=(c,c')$ in equation \eqref{Eq:Twisthopf}, for arbitrary $a,a',b,b',c,c'\in X$ and observe that the left hand side 
\begin{align*}
\bigvee_{m,n\in X}\ct{F}\big( \st{(a,m)_{1}. (b,n)_{1},(c,c')_{1}}\big).\ct{F}\big(\st{ (m,a')_{1},(n,b')_{1}}\big)
\end{align*}
is equal to $1$ if and only if $(a',b',c')= F_{12}\Psi (a,b,c)$, while the right hand side
\begin{align*}
\bigvee_{m,n\in X}\ct{F}\big( \st{(a,a')_{1},(b,m)_{1}.(c,n)_{1}}\big).\ct{F}\big(\st{ (m,b')_{1},(n,c')_{1}}\big)
\end{align*}
is equal to $1$ if and only if $(a',b',c')= F_{23}\Phi (a,b,c)$. Hence, \eqref{Eq:T1} holds. One the other hand, by the relations imposed in $H_{\omega}$, we have that 
$$\ct{F} \big( \st{(x,a)_{1},(y,b)_{1}. (z,c)_{1}}\big)= \ct{F} \big( \st{\sigma_{x}(y), \sigma_{a}(b) )_{1}.(\gamma_{y}(x), \gamma_{b}(a))_{1}. (z,c)_{1}}\big)$$
This directly implies that $r_{12}\Psi =  \Psi r_{12}$ and that \eqref{Eq:T3} holds. A symmetric argument implies that \eqref{Eq:T2} holds. 
\end{proof}

Note that the converse of the above statement is not true i.e. not every twist on a YBE solution can be realised as a co-twist on $H_{\omega}$. Firstly, co-twists on $H_{\omega}$ are built from twists on the underlying YBE solution which lift to its universal skew brace and we have already seen counterexamples for such situations in Example \ref{E:Universal}. Additionally, for a twist $(F,\Phi,\Psi)$ to extend to a co-twist on $H_{\omega}$, we require a lot more data. If we call the multiplication of $n$ generators of $H_{\omega}$, a word of order $n$, then the triple $(F,\Phi,\Psi)$ only determine the action of $\ct{F}$ on pairs of words of orders $(1,1)$, $(1,2)$ and $(2,1)$, respectively. For a co-twist $\ct{F}$ on $H_{\omega}$, we need the additional datum of its action on pairs of words of any order, which will have to be compatible with the triple via \eqref{Eq:Twisthopf}. Section III of \cite{kulish2000twisting} provides a full discussion of this in the linear setting, while \cite{doikou2021set} also addresses a weaker notion of twists on FRT-type algebras coming from the linearisation of involutive set-theoretical YBE solutions. The new ingredient for extending a twist to the $\SL$-FRT construction, is how $\ct{F}$ is defined on generators of the form $(x,y)_{2}$ in $H_{\omega}$. These actions will be uniquely determined by the action of $\ct{F}$ on generators $(x,y)_{1}$, but will impose a \emph{non-degeneracy} condition on $F$, as discussed in our comments after Theorem \ref{T:TwistYBE}.

In Section 5.2 of \cite{ghobadi2020skew}, it was also shown that any group  $(G,m,e)$ with a braiding operator $r$ appears as the remnant of a CQHA in $\SL$. The relevant Hopf algebra was denoted by $H_{\omega_{m}}$ and is the quotient of the $\SL$-FRT algebra of the underlying YBE solution $(G,r)$, with additional relations $(e,e)_{1}= 1= (e,e)_{2}$ and 
 \begin{align}
 (a.b,c)_{1}= \bigvee_{d,f\in G\ \text{with}\ d.f=c}\st{(a,d )_{1}. (b,f)_{1}}\quad \text{ and } \quad (a,b.c)_{2}= \bigvee_{d,f\in G\ \text{with}\ d.f=a} \st{(f,c)_{2}.(d,b )_{2}}\label{Eq:Hwm}
 \end{align}
for any $a,b,c\in G$. By Theorem \ref{T:AnyCQHA}, any co-twist on $H_{\omega_{m}}$ induces a twist on the skew brace $G$. Alternatively, since $H_{\omega_{m}}$ is defined in a similar fashion to $H_{\omega}$, one could define the triple $(F,\Phi,\Psi)$ as done in Theorem \ref{T:TwistHw}. Nevertheless, it's easy to observe that by definition both methods would provide the same triple. As in the proof of Theorem \ref{T:TwistHw}, the latter method provides a straightforward way to derive the definitions in Section \ref{S:DriTwis}, directly by studying the $\SL$-FRT algebras. As we will see in the next section, this is the key advantage which the $\SL$-FRT algebras posses over the classically considered groups with unique factorisation, when it comes to the study of set-theoretical YBE solutions.

\section{Co-twists and Groups with Unique Factorisation}\label{S:Groups}
In this section, we will classify co-twist on dualizable Hopf algebras in $\SL$ and thereby describe the induced Drinfeld twists on their remnants. As described in Section 4.1 of \cite{ghobadi2020skew}, dualizable Hopf algebras in $\SL$ are all of the form $\Pl{G}$, where $G$ is a group with unique factorization $G=G_{+}.G_{-}$ and the Hopf algebra structure of  $\Pl{G}$ is given as the bicrossproduct of the group algebra $\Pl{G_{-}}$ and the function algebra $\Pl{G_{+}}$. This theory was developed first in \cite{LYZ1} and quasitriangular structures on these Hopf algebras were classified in \cite{LYZ2}. In this section, we will adapt the notation of \cite{LYZ2} on groups with unique factorisation, which are also called matched pairs of groups. We refer the reader to Section 4.1 of \cite{ghobadi2020skew} for additional details on the re-phrasing of the results of \cite{LYZ1,LYZ2} in the setting of $\SL$. 

Throughout this section we will assume $G=G_{+}.G_{-}$ is a group with unique factorisation unless stated otherwise and $\Pl{G}$ denotes the relevant Hopf algebra in $\SL$. 
\begin{lemma}\label{L:ConvInv} Convolution invertible $\ct{F}: \Pl{G} \tn \Pl{G}\rightarrow \Pl{1}$ i.e. $\ct{F}$ which satisfy \eqref{Eq:ConvInv} are in bijection with maps $\Theta:G_{-}\times G_{-}\rightarrow G_{+}\times G_{+}$ such that 
\begin{equation}\label{Eq:TwistG+G-}
F_{\Theta} :( g_{-},h_{-}) \rightarrow \left( \prescript{\Theta_{1}( g_{-},h_{-})}{}{g_{-}}, \prescript{\Theta_{2}( g_{-},h_{-})}{}{h_{-}}\right)
\end{equation} 
defines a bijection from $G_{-}\times G_{-}$ to itself, where we denote $\Theta ( g_{-},h_{-})= \big(\Theta_{1}( g_{-},h_{-}), \Theta_{2}( g_{-},h_{-})\big)$.
\end{lemma} 
\begin{proof} Given a convolution invertible $\ct{F}$, we can define $\Theta (g_{-},h_{-})= (a_{+},b_{+})$ for elements where $\ct{F}^{-1}(a_{+}g_{-},b_{+}h_{-})= 1$. First, we must show that $\Theta$ is well-defined. By \eqref{Eq:ConvInv}, for any pair $g_{-},h_{-}\in G_{-}$ there exists at lease one pair $a_{+},b_{+}\in G_{+}$ such that 
$$\ct{F}^{-1}(a_{+}g_{-},b_{+}h_{-})= 1=\ct{F}\left(a^{-1}_{+}\big(\prescript{a_{+}}{}{g_{-}}\big),b^{-1}_{+}\big( \prescript{b_{+}}{}{h_{-}}\big)\right)$$
and if $\ct{F}\left( k_{+}\big(\prescript{a_{+}}{}{g_{-}}\big),l_{+}\big(\prescript{b_{+}}{}{h_{-}}\big)\right)=1$, then $k_{+}=a_{+}^{-1}$ and $l_{+}=b_{+}^{-1}$. A symmetric statement holds for $\ct{F}$ and for any $g_{-},h_{-}\in G_{-}$, there exist at least one pair $m_{+},n_{+}\in G_{+}$ such that $\ct{F}(m_{+}g_{-},n_{+}h_{-})= 1=\ct{F}^{-1}\left(m^{-1}_{+}\big(\prescript{m_{+}}{}{g_{-}}\big),n^{-1}_{+}\big(\prescript{n_{+}}{}{h_{-}}\big)\right)$. By the latter statement, for the pair $\prescript{a_{+}}{}{g_{-}}, \prescript{b_{+}}{}{h_{-}}$ there exists at least one pair $k_{+},l_{+}\in G_{+}$ satisfying $\ct{F}\left( k_{+}\big(\prescript{a_{+}}{}{g_{-}}\big),l_{+}\big(\prescript{b_{+}}{}{h_{-}}\big)\right)=1$ and by the first statement we have that $k_{+}=a_{+}^{-1}$ and $l_{+}=b_{+}^{-1}$. Hence, if $\ct{F}^{-1}(a_{+}g_{-},b_{+}h_{-})= 1$, then $\ct{F}\left(a^{-1}_{+}\big(\prescript{a_{+}}{}{g_{-}}\big),b^{-1}_{+}\big( \prescript{b_{+}}{}{h_{-}}\big)\right)=1$ and the symmetric statement holds for $\ct{F}$.
 
Now we show that for any $g_{-},h_{-}\in G_{-}$, the pair $a_{+},b_{+}$ satisfying $\ct{F}^{-1}(a_{+}g_{-},b_{+}h_{-})= 1$ is unique. Assume that the pair $k_{+},l_{+}$ also satisfies $\ct{F}^{-1}(k_{+}g_{-},l_{+}h_{-})= 1 $, then $\ct{F}\left(k^{-1}_{+}\big(\prescript{k_{+}}{}{g_{-}}\big),f^{-1}_{+}\big( \prescript{l_{+}}{}{h_{-}}\big)\right)=1$ and thereby 
$$ \ct{F}^{-1}(a_{+}g_{-},b_{+}h_{-}).\ct{F}\left(k^{-1}_{+}\big(\prescript{k_{+}}{}{g_{-}}\big),l^{-1}_{+}\big( \prescript{l_{+}}{}{h_{-}}\big)\right)=1= \epsilon (a_{+}k^{-1}_{+}g_{-} ) .\epsilon (b_{+}l^{-1}_{+}h_{-}) $$
Hence, $k_{+}=a_{+}$ and $l_{+}=b_{+}$ and thereby $\Theta$ is a well-defined map. In a symmetric way, we can define $\Theta':G_{-}\times G_{-}\rightarrow G_{+}\times G_{+}$ by $\Theta' (g_{-},h_{-})= (a_{+},b_{+})$ for $\ct{F}(a_{+}g_{-},b_{+}h_{-})= 1$. By \eqref{Eq:ConvInv}, it then follows that the maps 
\begin{align*}
F_{\Theta}(g_{-},h_{-})= \left( \tht{1}{g_{-}}{h_{-}}{g_{-}}, \tht{2}{g_{-}}{h_{-}}{h_{-}}\right), \quad F^{-1}_{\Theta}(g_{-},h_{-})= \left( \prescript{\Theta'_{1}( g_{-},h_{-})}{}{g_{-}}, \prescript{\Theta'_{2}( g_{-},h_{-})}{}{h_{-}}\right)
\end{align*}
are inverses. 
\end{proof}
Before we resume, recall from Section 4.1 of \cite{ghobadi2020skew} that the remnant of $\Pl{G}$ is the group $G_{-}$ and co-quasitriangular structures on $\Pl{G}$ are in bijection with pairs of maps $\eta, \xi: G_{-}\rightarrow G_{+}$ satisfying additional conditions. Now we calculate the induced twist data on $G_{-}$ by Theorem \ref{T:AnyCQHA}. Note that co-twist on a Hopf algebra are defined independent of any co-quasitriangular structure and should not take the braiding operator into consideration. While the $\SL$-FRT algebras depended directly on the braiding, this is not the case for $\Pl{G}$. In this case, the twists induced on $G_{-}$ by Theorem \ref{T:AnyCQHA} will automatically be compatible with any of the possible braiding operators which arise from co-quasitriangular structures on $\Pl{G}$. We will denote the units of $G_{-}$ and $G_{+}$ by $e_{-}$ and $e_{-}$, respectively.
\begin{thm}\label{T:TwistG+G-} Co-twists on $\Pl{G}$ are in bijection with maps $\Theta:G_{-}\times G_{-}\rightarrow G_{+}\times G_{+}$ such that $F_{\Theta}$, as defined in Lemma \ref{L:ConvInv}, is invertible and $\Theta$ satisfies $\Theta_{2}(e_{-},a_{-} )=e_{+}=\Theta_{1}(a_{-}, e_{-})$ and the following conditions:
\begin{align} 
\Theta_{1}\left(  \prescript{\Theta_{1} ( ab,c)}{}{a},\prescript{\Theta_{1} ( ab,c)^{a}}{}{b}\right).\Theta_{1} ( ab,c) =\Theta_{1}(a,bc )\quad \quad &
\\\Theta_{2}\left(  \prescript{\Theta_{1} ( ab,c)}{}{a},\prescript{\Theta_{1} ( ab,c)^{a}}{}{b}\right).\Theta_{1} ( ab,c)^{a}=\Theta_{1}\left( \prescript{\Theta_{2}(a,bc )}{}{b},  \prescript{\Theta_{2}(a,bc )^{b}}{}{c}\right) &.\Theta_{2}(a,bc )
\\\Theta_{2} ( ab,c)= \Theta_{2}\left( \prescript{\Theta_{2}(a,bc )}{}{b}, \prescript{\Theta_{2}(a,bc )^{b}}{}{c}\right).\Theta_{2}(a,bc )^{b}\quad  &
\end{align}
where $a,b,c\in G_{-}$ and the equations above hold in $G_{+}$.
\end{thm}
\begin{proof} First we recall that the unit element in $\Pl{G}$ is $\vee_{g_{+}\in G_{+}}g_{+}$ and thereby \eqref{Eq:Twistcounit} would translate to $\vee_{g_{+}\in G_{+}}\ct{F}^{-1}(g_{+}, a_{+}a_{-}) = 1$ if and only if $a_{+}=e_{+}\in G_{+}$. Equivalently, $\Theta_{2}(e_{-},a_{-} )=e_{+}$. A symmetric argument show that $\Theta_{1}(a_{-}, e_{-})=e_{+}$. Since $\Theta$ is formulated in terms of $\ct{F}^{-1}$, we will look at the equivalent condition to \eqref{Eq:Twisthopf} in terms of $\ct{F}^{-1}$, \eqref{Eq:Twisthopf2}. Let $a,b,c\in G$, then the left hand side of equation \eqref{Eq:Twisthopf2} expands as follows:
\begin{align*}
\bigvee_{g,h\in G_{+}}\ct{F}^{-1}&\left( a_{+}g^{-1}\big( \prescript{g}{}{a_{-}}\big),b_{+}h^{-1}\big( \prescript{h}{}{b_{-}}\big)\right).\ct{F}^{-1}\left( ga_{-}. hb_{-},c_{+}c_{-}\right)
\\ &=\bigvee_{g\in G_{+}}\ct{F}^{-1}\left( a_{+}g^{-1}\big( \prescript{g}{}{a_{-}}\big),b_{+}(g^{a_{-}})^{-1}\big( \prescript{g^{a_{-}}}{}{b_{-}}\big)\right).\ct{F}^{-1}\left( ga_{-}b_{-},c_{+}c_{-}\right)
\end{align*}
The equation is equal to $1$ if and only if $a_{+}= \Theta_{1}\left(  \prescript{g}{}{a_{-}},\prescript{g^{a_{-}}}{}{b_{-}}\right).g$, $b_{+}= \Theta_{2}\left(  \prescript{g}{}{a_{-}},\prescript{g^{a_{-}}}{}{b_{-}}\right).g^{a_{-}}$, $c_{+}= \Theta_{2} ( a_{-}b_{-},c_{-})$ and $g=\Theta_{1} ( a_{-}b_{-},c_{-})$. While the right hand side of \eqref{Eq:Twisthopf2} expands as
\begin{align*}
\bigvee_{g,h\in G_{+}}\ct{F}^{-1}&\left(b_{+}g^{-1}\big( \prescript{g}{}{b_{-}}\big), c_{+}h^{-1}\big( \prescript{h}{}{c_{-}}\big)\right) \ct{F}^{-1}\left( a_{+}a_{-},gb_{-}.hc_{-}\right)
\\ &=\bigvee_{g\in G_{+}}\ct{F}^{-1}\left(b_{+}g^{-1}\big( \prescript{g}{}{b_{-}}\big), c_{+}(g^{b_{-}})^{-1}\big( \prescript{g^{b_{-}}}{}{c_{-}}\big)\right) \ct{F}^{-1}\left( a_{+}a_{-},gb_{-}c_{-}\right)
\end{align*} 
which will be equal to $1$ if and only if $b_{+}=\Theta_{1}\left( \prescript{g}{}{b_{-}},  \prescript{g^{b_{-}}}{}{c_{-}}\right) .g$, $c_{+}= \Theta_{2}\left( \prescript{g}{}{b_{-}}, \prescript{g^{b_{-}}}{}{c_{-}}\right).g^{b_{-}}$, $a_{+}= \Theta_{1}(a_{-},b_{-}c_{-} )$ and $g=\Theta_{2}(a_{-},b_{-}c_{-} )$. The equalities in Theorem then simply arise as equating the value of $a_{+},b_{+}$ and $c_{+}$ from the left and right hand side of \eqref{Eq:Twisthopf2}, respectively. 
\end{proof}
The works \cite{LYZ1,LYZ2} focus on linear Hopf algebras with \emph{positive bases} and show that all such Hopf algebras come from finite groups with unique factorisation, and classify all \emph{positive} quasitriangular structures on them. The proofs in these works consist of two parts. The first part is finding an isomorphism so that the structural maps don't have any scalar factor, which basically translates the problem into a problem about Hopf algebras on free lattices. The second part of their proofs classify such structures and fully translate into $\SL$, see Section 4.1 of \cite{ghobadi2020skew}. Hence, we expect that the above classification can be re-phrased as classifying all \emph{positive} Drinfeld co-twists on linear Hopf algebras with positive bases. 
\begin{corollary}\label{C:G+G-} Given a co-twist on $\Pl{G}$ defined by a map $\Theta$, we obtain a twist $(F_{\Theta},\Phi_{\Theta},\Psi_{\Theta})$ on the group $G_{-}$, where $F_{\Theta}$ is as defined in Lemma \ref{L:ConvInv} and the other two maps are defined by 
\begin{align}
\Phi_{\Theta} (a,b,c) &=\left(\tht{1}{a}{bc}{a}, \tht{2}{a}{bc}{b},\prescript{\Theta_{2}(a,bc)^{b}}{}{c}\right) \label{Eq:PhiTheta}
\\ \Psi_{\Theta} (a,b,c) &=\left(\tht{1}{ab}{c}{a}, \prescript{\Theta_{1}(ab,c)^{a}}{}{b},\tht{2}{ab}{c}{b}\right) \label{Eq:PsiTheta}
\end{align}
for $a,b,c\in G_{-}$. In particular, the twist on $G_{-}$ is independent of $\eta$ and $\xi$  and the braiding operator on $G_{-}$.
\end{corollary}
\begin{proof} Recall from Section 4.1 of \cite{ghobadi2020skew} that $\pi:G\rightarrow G_{-}$ is defined by $\pi (a_{+}a_{-})=\delta_{a_{+},e_{+}} a_{-}$ for $a_{+}\in G_{+}$ and $a_{-}\in G_{-}$. Hence, for $a,b\in G_{-}$ we have 
\begin{align*}
F(a&,b)= \bigvee_{g,h,k,l\in G_{+}}\ct{F}\left(g^{-1}(\prescript{g}{}{a}),h^{-1}(\prescript{h}{}{b}) \right).\big(\pi (gk^{-1}(\prescript{k}{}{a})),\pi( hl^{-1}(\prescript{l}{}{b}))\big).\ct{F}^{-1}\left( ka,lb \right)
\\ &= \bigvee_{g,h\in G_{+}}\ct{F}\left(g^{-1}(\prescript{g}{}{a}),h^{-1}(\prescript{h}{}{b}) \right).\big(\pi (\prescript{g}{}{a}),\pi( \prescript{h}{}{b})\big).\ct{F}^{-1}\left( ga,hb \right) = \big(\tht{1}{a}{b}{a}, \tht{1}{a}{b}{b}\big)= F_{\Theta}(a,b)
\end{align*}
A similar calculation follows for $\Phi_{\Theta}$ and $\Psi_{\Theta}$, using the fact that $(\prescript{h}{}{b})\left(\prescript{h^{b}}{}{c}\right)=\prescript{h}{}{(bc)}$ for $b,c\in G_{-}$ and $h\in G_{+}$ [Equation (1) in \cite{LYZ2}]:
\begin{align*}
\Phi (a,b,c )=& \bigvee_{g,h,f,k,l,m\in G_{+}}\ct{F}\left(g^{-1}(\prescript{g}{}{a}),h^{-1}(\prescript{h}{}{b}).f^{-1}(\prescript{f}{}{c}) \right).\big(\pi (gk^{-1}(\prescript{h}{}{a})),\pi( hl^{-1}(\prescript{l}{}{b})),\pi(fm^{-1}(\prescript{m}{}{b}))\big)
\\&.\ct{F}^{-1}\left(ka,lb.mc \right)
\\=& \bigvee_{g,h,f\in G_{+}}\ct{F}\left(g^{-1}(\prescript{g}{}{a}),h^{-1}(\prescript{h}{}{b}).f^{-1}(\prescript{f}{}{c}) \right).\big(\pi (\prescript{g}{}{a}),\pi( \prescript{h}{}{b}),\pi(\prescript{f}{}{b})\big).\ct{F}^{-1}\left(ga,hb.fc \right)
\\=& \bigvee_{g,h\in G_{+}}\ct{F}\left(g^{-1}(\prescript{g}{}{a}),h^{-1}(\prescript{h}{}{b})\left(\prescript{h^{b}}{}{c}\right) \right).\left(\pi (\prescript{g}{}{a}),\pi( \prescript{h}{}{b}),\pi\left(\prescript{h^{b}}{}{c}\right)\right).\ct{F}^{-1}\left(ga,hbc \right)
\\=& \left(\tht{1}{a}{bc}{a}, \tht{2}{a}{bc}{b},\prescript{\Theta_{2}(a,bc)^{b}}{}{c}\right)
\\ \Psi (a,b,c )=& \bigvee_{g,h,f,k,l,m\in G_{+}}\ct{F}\left(g^{-1}(\prescript{g}{}{a}).h^{-1}(\prescript{h}{}{b}),f^{-1}(\prescript{f}{}{c}) \right).\big(\pi (gk^{-1}(\prescript{k}{}{a})),\pi( hl^{-1}(\prescript{l}{}{b})),\pi(fm^{-1}(\prescript{m}{}{b}))\big)
\\&.\ct{F}^{-1}\left(ka.lb,mc \right)
\\=& \bigvee_{g,h,f\in G_{+}}\ct{F}\left(g^{-1}(\prescript{g}{}{a}).h^{-1}(\prescript{h}{}{b}),f^{-1}(\prescript{f}{}{c}) \right).\big(\pi (\prescript{g}{}{a}),\pi( \prescript{h}{}{b}),\pi(\prescript{f}{}{b})\big).\ct{F}^{-1}\left(ga.hb,fc \right)
\\=& \bigvee_{g,f\in G_{+}}\ct{F}\left(g^{-1}(\prescript{g}{}{a})\left(\prescript{g^{a}}{}{b}\right),f^{-1}(\prescript{f}{}{c}) \right).\left(\pi (\prescript{g}{}{a}),\pi\left( \prescript{g^{a}}{}{b}\right),\pi(\prescript{f}{}{b})\right).\ct{F}^{-1}\left(gab,fc \right)
\\= &\left(\tht{1}{ab}{c}{a}, \prescript{\Theta_{1}(ab,c)^{a}}{}{b},\tht{2}{ab}{c}{c}\right)
\end{align*} 
\end{proof}
As mentioned before the proof of this result co-quasitrianngular structures on $\Pl{G}$ are classified by maps $\eta,\xi:G_{-}\rightarrow G_{+}$ satisfying conditions of Proposition 1 in \cite{LYZ2}. Given such a pair we obtain an induced braiding operator on $G_{-}$ defined by $r:(g,h)\mapsto (\prescript{\eta(g)}{}{h},g^{\xi(h)})$. Theorem \ref{T:AnyCQHA} tells us that the triples defined above produce a twist on $(G_{-},r)$ regardless of the choice of $\eta $ and $\xi$. This can be verified directly as well and follows from the axioms which $\eta $ and $\xi$ satisfy. 

A skew brace $(G_{-},r)$ coming from a matched pair of groups can also arise as the remnant of a $\SL$-FRT algebra. However, given a twist on $(G_{-},r)$ the obstruction of extending the twist to a co-twist on the two CQHAs is very different. In the FRT case, we were missing additional information, whereas here the sole obstruction to extending a twist, $F$, to a co-twist on $\Pl{G}$ is whether it can be written as $F_{\Theta}$ for some $\Theta$, which in return involves the group $G_{+}$ and its action on $G_{-}$. 

The simplest case of groups with unique factorisation are provided by groups with braiding operators themselves. If $(G,r)$ is of the latter form, then we can define $G_{-}=G_{+}=G$ with actions 
$$\prescript{g_{+}}{}{g_{-}}= \sigma_{g_{+}}(g_{-}),\quad g_{+}^{g_{-}}= \gamma_{g_{-}}(g_{+}), \quad \prescript{g_{-}}{}{g_{+}}= \tau_{g_{+}}(g_{-}), \quad g_{-}^{g_{+}}= \rho_{g_{+}}(g_{-})$$
where $r^{-1}(x,y) =(\tau_{x}(y),\rho_{y}(x))$. Thereby, we obtain a group with unique factorisation $G_{+}.G_{-}$ on the underlying set $G\times G$. The braiding operator on $G_{-}$ is then recovered by $\eta=\xi= \id_{G}$. Hence, we can re-write Theorems \ref{T:TwistG+G-} and Corollary \ref{C:G+G-} solely in terms of $(G,r)$. 

\begin{corollary}\label{C:G-} Let $(G,m,e)$ be a group with a braiding operator. If $\Theta :G^{2}\rightarrow G^{2}$ is a map such that $F_{\Theta}:G^{2}\rightarrow G^{2}$ defined by $F_{\Theta}(g,h)= (\sigma_{\Theta_{1}(g,h)}(h), \sigma_{\Theta_{2}(g,h)}(g))$ is a bijection and satisfies $\Theta_{2}(e_{-},a_{-} )=e_{+}=\Theta_{1}(a_{-}, e_{-})$ and
\begin{align}
\Theta_{1}\left(  \sigma_{\Theta_{1} ( ab,c)}(a),\sigma_{\gamma_{a}(\Theta_{1} ( ab,c))}(b)\right).\Theta_{1} ( ab,c) =\Theta_{1}(a,bc )\quad \quad \label{Eq:Theta1}&
\\\Theta_{2} ( ab,c)= \Theta_{2}\left( \sigma_{\Theta_{2}(a,bc )}(b), \sigma_{\gamma_{b}(\Theta_{2}(a,bc ))}(c)\right).\gamma_{b}\big(\Theta_{2}(a,bc )\big)\quad  \label{Eq:Theta2}&
\\\Theta_{2}\left(  \sigma_{\Theta_{1} ( ab,c)}(a),\sigma_{\gamma_{a}(\Theta_{1} ( ab,c))}(b)\right).\gamma_{a}\big(\Theta_{1} ( ab,c)\big)=\Theta_{1}\left( \sigma_{\Theta_{2}(a,bc )}(b),  \sigma_{\gamma_{b}(\Theta_{2}(a,bc ))}(c)\right) &.\Theta_{2}(a,bc )\label{Eq:Theta3}
\end{align} 
for $a,b,c\in G$, then we obtain a new group $(G,mF^{-1}_{\Theta},e)$ with a braiding operator $F_{\Theta} r F_{\Theta}^{-1}$.
\end{corollary}

We have already seen examples of twists which can be written of the above form in Section \ref{S:DriTwis}. If we define $\Theta(x,y)= (e,x)$, then $F_{\Theta}(x,y)=(x,\sigma_{x}(y))$ is bijective by definition, while \eqref{Eq:Theta1} and \eqref{Eq:Theta2} hold in a trivial manner and \eqref{Eq:Theta2} reduces to $ab=\sigma_{a}(b)\gamma_{b}(a)$. With this choice of $\Theta$, we recover the triple $F_{\Theta}(x,y)= ( x,\sigma_{x}(y))$, $\Phi_{\Theta}(x,y,z)=(x, \sigma_{x}(y), \sigma_{\gamma_{y}(x)}(z))$ and $\Psi_{\Theta} (x,y,z)= (x,y,\sigma_{xy}(z))$ of Example \ref{E:Theta}. 
The reader should also note that the inverse twists to this family of twists cannot be written in the form of Corollary \ref{C:G-}. By Theorem \ref{T:SkwGroupoid}, any twist of the mentioned form has an inverse twist $(F',\Phi',\Psi')$ with $F'=F^{-1}$ on the group $(G,\star)$ with the trivial braiding operator $(x,y)\mapsto (y, y^{\star}\star x\star y)$. Due to the trivial braiding operator on $(G,\star, \star)$, the twists coming from Corollary \ref{C:G-} on this skew brace are independent of the choice of $\Theta$ and define the same twist given by $F_{\Theta}(x,y)=(y,x) $. Given this limitation, we can conclude that the general theory Drinfeld twists described in Section \ref{S:DriTwis}, truly goes beyond the study of twist on Hopf algebras coming from matched pairs of groups. 

\section{Concluding Remarks and Outlook}\label{S:Final}
While in \cite{ghobadi2020skew}, we provided the theory which relates skew braces with Hopf algebra objects in a suitable category, $\SL$, the introduction of Drinfeld twists in this work is the first step in advancing this theory by translating Hopf algebraic constructions, such as Drinfeld twists, into constructions which are applicable to skew braces in a compatible way. There are a number of other constructions on CQHAs which should have fruitful applications to the study of skew braces, such as the theory of co-double bosonisation \cite{aziz2019co}. The studies in Sections \ref{S:SL} and \ref{S:Groups}, also fully demonstrate the advantage of applying Hopf algebraic techniques to general Hopf algberas in $\SL$ and observing the effect on their remnants compared to studying these technique on Hopf algebras coming from matched pairs of groups. 

There are several questions which remain open with regards to the discussion of Drinfeld twists on skew brace. The first natural question is whether every twist on a skew brace can be extended to a co-twist on a CQHA in $\SL$, which has this skew brace as its remnant. The second part of this question is whether every twist can be extended canonically to a co-twist on the $\SL$-FRT algebra $H_{\omega_{m}}$ related to the skew brace. As we mentioned in Section \ref{S:SL}, more data is needed to define a co-twist on $H_{\omega_{m}}$, but there might be a natural choice for selecting this data. The computational aspects of this theory might also be of interest. By Example \ref{E:Theta}, any skew brace can be obtained as a twist from the trivial skew brace of its additive group. Hence, one could write an algorithm for to classify twists on such objects and thereby all skew braces with the same additive group. It is not however clear whether such an algorithm or variations of it could have a computational advantage to the algorithm presented in Section 5 of \cite{guarnieri2017skew}.
\bibliographystyle{plain}
\bibliography{Hopf}

\end{document}